\documentclass[12pt,reqno]{amsart}

\usepackage{amsfonts,amsmath,amsthm}
\usepackage{amssymb,epsfig}
\usepackage{cases}
\usepackage[usenames]{color}
\usepackage{enumerate} 

\usepackage{color} 
\definecolor{vert}{rgb}{0,0.6,0}

\theoremstyle{plain}
\newtheorem{thm}{Theorem}[section]

\newtheorem{defn}{Definition}

\newtheorem{lem}[thm]{Lemma}

\newtheorem{prop}[thm]{Proposition}
\theoremstyle{remark}
\newtheorem{rem}{\bf{Remark}}

\newcommand{\N}{\mathbb{N}}
\newcommand{\R}{\mathbb{R}}

\newcommand{\cA}{\mathcal{A}}

\newcommand{\cM}{\mathcal{M}}

\newcommand{\BUC}{{\rm BUC\,}}
\newcommand{\USC}{{\rm USC\,}}
\newcommand{\LSC}{{\rm LSC\,}}
\newcommand{\UC}{{\rm UC\,}}
\newcommand{\Li}{L^{\infty}}
\newcommand{\W}{W^{1,\infty}}

\newcommand{\bO}{\partial\Omega}
\newcommand{\cO}{\overline\Omega}
\newcommand{\Q}{\Omega\times(0,\infty)}

\newcommand{\bQ}{\partial\Omega\times(0,\infty)}

\newcommand{\cQ}{\overline\Omega\times[0,\infty)}

\newcommand{\QT}{Q_T}

\newcommand{\al}{\alpha}
\newcommand{\gam}{\gamma}
\newcommand{\del}{\delta}
\newcommand{\ep}{\varepsilon}

\newcommand{\sig}{\sigma}
\newcommand{\om}{\omega}

\newcommand{\Om}{\Omega}

\newcommand{\ol}{\overline}
\newcommand{\ul}{\underline}
\newcommand{\pl}{\partial}

\newcommand{\dist}{{\rm dist}\,}

\def\ue{u^\varepsilon}
\def\u1e{u_{1,\ep}}
\def\limssup{\mathop{\rm limsup\!^*}}
\def\limiinf{\mathop{\rm liminf_*}}

\begin{document}
\title[Large-Time Asymptotics For Boundary-Value Problems]
{A PDE approach to large-time asymptotics
for boundary-value problems for
nonconvex Hamilton-Jacobi Equations
}

\author[G. BARLES and H. MITAKE]
{Guy BARLES and Hiroyoshi MITAKE}

\address[G. Barles]
{Laboratoire de Math\'ematiques et Physique 
Th\'eorique (UMR CNRS 6083), F\'ed\'eration Denis Poisson, 
Universit\'e de Tours, 
Place de Grandmont, 
37200 Tours, FRANCE}
\email{barles@lmpt.univ-tours.fr}
\urladdr{http://www.lmpt.univ-tours.fr/~barles}

\address[H. Mitake]
{Department of Applied Mathematics,
Graduate School of Engineering
Hiroshima University
Higashi-Hiroshima 739-8527, Japan}
\email{mitake@amath.hiroshima-u.ac.jp}


\thanks{This work was partially supported by the ANR project ``Hamilton-Jacobi et th\'eorie KAM faible'' (ANR-07-BLAN-3-187245) and by the Research Fellowship (22-1725)
for Young Researcher from JSPS}

\date{\today}

\begin{abstract}
We investigate the large-time behavior of three types of initial-boundary value problems for 
Hamilton-Jacobi Equations with nonconvex Hamiltonians. We consider the 
Neumann or oblique boundary condition, 
the state constraint boundary condition and Dirichlet boundary condition. We establish general convergence results for viscosity solutions to asymptotic solutions as time goes to infinity via an approach based on PDE techniques. These results are obtained not only under general conditions on the Hamiltonians but also under weak conditions on the domain and the oblique direction of reflection in the Neumann case.
\end{abstract}

\maketitle


\section{Introduction and Main Results}

In this paper we investigate the large time behavior of viscosity solutions of the initial-boundary value problems for Hamilton-Jacobi Equations which we write under the form
\begin{numcases}
{\textrm{(IB)} \hspace{1cm}}
u_t+H(x,Du)= 0 & in $\Q$,\label{ib-1} \\
B(x,u,Du)=0 & on $\bQ$,\label{ib-2}\\
u(x,0)=u_{0}(x) & on $\cO$,\label{ib-3}
\end{numcases}
where 
$\Om$ is a bounded domain of $\R^{N}$, 
$H=H(x,p)$ is a given real-valued continuous function on $\cQ$, 
which is {\em coercive}, i.e., 
\begin{itemize}
\item[{\rm(A1)}]\hspace{1cm} $\displaystyle \lim_{r\to+\infty}\inf\{H(x,p)\mid x\in\cO, 
|p|\ge r \}=+\infty. 
$
\end{itemize} 
Here $B$ and $u_{0}$ are given real-valued continuous functions 
on $\bO\times\R\times\R^{N}$ and 
$\cO$, respectively. 
The solution $u$ is a real-valued function on $\cQ$ and 
we respectively denote by $u_t:=\pl u/\pl t$ and 
$Du:=(\pl u/\pl x_1,\ldots,\pl u/\pl x_N)$ its time derivative and gradient with respect to the space variable. 
We are dealing only with viscosity solutions of Hamilton-Jacobi equations in this paper and thus the term ``viscosity" may be omitted henceforth. We also point out that the boundary conditions have to be understood in the viscosity sense : we refer the reader to the beginning of Appendix where this definition is recalled.

We consider three types of boundary conditions (BC in short)
\begin{equation}\label{neumann}
\hbox{(Neumann or oblique BC)}\quad B(x,r,p)=\gam(x)\cdot p -g(x), 
\end{equation}
where $g$ is a given real-valued continuous function on $\bO$ and 
$\gam:\cO\to\R^{N}$ is a continuous vector field which 
is oblique to $\bO$, i.e., 
\begin{equation}\label{oblique}
n(x)\cdot\gam(x)>0 
\quad\textrm{for any} \ x\in\bO, 
\end{equation}
where 
$n(x)$ is the outer unit normal vector at $x$ to 
$\bO$,
\begin{equation}\label{sc}
\hbox{(state constraint BC) }\quad
B(x,r,p)= -1 \ 
\textrm{for all} \ (x,r,p)\in\cO\times\R\times\R^{N},
\end{equation}
and finally
\begin{equation}\label{dirichlet}
\hbox{(Dirichlet BC) }\quad
B(x,r,p)=r-g(x).  
\end{equation}
We respectively denote Problem (IB) with boundary conditions \eqref{neumann}, \eqref{sc} 
and \eqref{dirichlet} by (CN), (SC) and (CD).

We impose different regularity assumptions on $\bO$ depending on the boundary condition we are treating. 
When we consider Neumann/oblique derivative problems, 
we \textit{always} assume that $\Om$ is a domain with a $C^{1}$-boundary. 
When we consider state constraint or Dirichlet problems,
we \textit{always} assume that $\Om$ is a domain with 
a $C^{0}$-boundary. 
Moreover when we consider Dirichlet problems, we assume 
a compatibility condition on the initial value $u_{0}$ 
and the boundary value $g$, i.e.,  
\begin{equation}\label{comp-cond}
u_{0}(x)\le g(x) \ \textrm{for all} \ x\in\bO. 
\end{equation}
Since we are going to assume $\gamma$ in \eqref{neumann} to 
be only continuous (and not Lipschitz continuous 
as it is classically the case), 
we have to solve Problem (IB) with very weak conditions. 
This is possible as a consequence of the coercivity assumption on $H$ 
since the solutions are expected to be in $W^{1,\infty}$, which denotes 
the set of bounded functions whose the first distributional derivatives 
are essentially bounded. 
We postpone the proof of the existence, uniqueness and regularity of the solution $u$ of (IB) to Appendix since the proofs are a little bit technical and the main topic of this paper concerns the asymptotic behavior as $t\to +\infty$ of the solutions. The reader may first assume that $\bO$ and $\gamma$ are regular enough to concentrate on this asymptotic behavior and then consider Appendix for the generalization to less regular situations.

More precisely we have the following existence 
result for (IB).
\begin{thm}[Existence of Solutions of {\rm (IB)}]\label{thm:existence}
Assume that {\rm (A1)} holds and that $u_{0}\in\W(\Om)$. \\
{\rm (i)} Assume that $\Om$ is a bounded domain with a $C^1$-boundary and 
that $\gamma$, $g$ are continuous functions which satisfy \eqref{oblique}. 
There exists a unique solution $u$ of {\rm (CN)} 
which is continuous on $\cQ$ and in $W^{1,\infty}(\Om \times [0,\infty))$.\\
{\rm (ii)} Assume that $\Om$ is a bounded domain with a $C^0$-boundary.  
There exists a unique solution $u$ of {\rm (SC)} 
which is continuous on $\cQ$ and in $W^{1,\infty}(\Om \times [0,\infty))$.\\
{\rm (iii)} 
Assume that $\Om$ is a bounded domain with a $C^0$-boundary 
and that $g\in C(\bO)$ and $u_0$ which satisfy 
\eqref{comp-cond}.
There exists a unique solution $u$ of {\rm (CD)} 
which is continuous on $\cQ$ and in $W^{1,\infty}(\Om \times [0,\infty))$. 
\end{thm}

Throughout this work, we are going to use {\it anytime} the same assumptions as in Theorem~\ref{thm:existence} which clearly depend on the kind of problem we are considering. In order to simplify the statements of our results, we will say from now on the ``basic assumptions'' ((BA) in short) are satisfied if  {\rm (A1)} holds and if $\Om$ is a bounded domain with a $C^1$-boundary in the case of Neumann/oblique derivatives boundary conditions or $\Om$ is a bounded domain with a $C^0$-boundary in the case of state constraint or Dirichlet boundary conditions. When it is relevant, $u_{0}\in\W(\Om)$ is also assumed to be part of (BA) (results for time-dependent problems) and so are the continuity assumptions for $\gamma, g$ on $\bO$ as well as \eqref{oblique} and \eqref{comp-cond}.

Next we recall that the standard asymptotic behavior for solutions of Hamilton-Jacobi Equations is given by an \textit{additive eigenvalue} 
or \textit{ergodic problem} : for $t$ large enough, the solution $u(x,t)$ is expected to look like $-c t + v(x)$ where the constant $c$ and the function $v$ are solutions of
$$ H(x,Dv(x))= c \quad \hbox{in  }\Om\, ,$$
with suitable boundary conditions. A typical result, which was first proved in $\R^{N}$ for the periodic case by  
P.-L. Lions, G. Papanicolaou and S. R. S. Varadhan \cite{LPV}, is that there exists a unique constant $c$ for which this problem has a solution, 
while the solution $v$ may be non unique, even up to an additive constant. 
This non-uniqueness feature is 
a key difficulty in the study of the asymptotic behavior.

In the Neumann/oblique derivative or state constraint case, 
the additive eigenvalue or ergodic problem reads
\begin{numcases}
{{\rm (E)} \hspace{1cm}}
H(x,Dv(x))= a & in  $\Om$,\label{erg-1}\\
B(x,v(x),Dv(x))=0 & on $\bO$. \label{erg-2}
\end{numcases}
Here one seeks for a
pair $(v,a)$ of $v\in C(\cO)$ and $a\in\R$ 
such that $v$ is a solution of (E). 
If $(v,a)$ is such a pair, we call $v$ an \textit{additive eigenfunction} or 
\textit{ergodic function} 
and $a$ an \textit{additive eigenvalue} 
or \textit{ergodic constant}. 
We denote Problem (E) with \eqref{neumann}, \eqref{sc} 
by (E-N) and (E-SC). 

For the large-time asymptotics of solutions of (CD), as we see it later, there are different cases which are treated either by using (E-SC) 
or the stationary Dirichlet problems
\begin{numcases}
{{\rm (D)}_{a} \hspace{1cm}}
H(x,Dv(x))= a & in  $\Om$,\label{d-1}\\
v(x)=g(x) & on $\bO$ \label{d-2}
\end{numcases}
for $a\in\R$.
Since the existence of solutions of 
(E-N), (E-SC) and (D)$_{a}$ are playing a central role, we first state the following result.

\begin{thm}[Existence of Solutions of Additive Eigenvalue Problems]\label{thm:additive}
Assume that {\rm (BA)} holds. \\
{\rm (i)} There exists a solution 
$(v,c_{n})\in W^{1,\infty}(\Om)\times\R$ of {\rm (E-N)}. 
Moreover, the additive eigenvalue is unique and is represented by 
\begin{equation}\label{additive-const-neumann}
c_{n}=\inf\{a\in\R\mid 
{\rm (E}\textrm{-}{\rm N)} 
\ \textrm{has a subsolution}\}. 
\end{equation}
{\rm (ii)} 
There exists a solution $(v,c_{sc})\in \W(\Om)\times\R$ of {\rm (E-SC)}. 
An additive eigenvalue is unique and is represented by 
\begin{equation}\label{additive-const-state}
c_{sc}=\inf\{a\in\R\mid \eqref{erg-1} \ \textrm{has a subsolution}\}. 
\end{equation}
{\rm (iii)}
There exists a solution 
$v\in \W(\Om)$ of {\rm (D)$_{a}$}
if and only if $a\ge c_{sc}$.  
\end{thm}
Existence results of solutions of (E) have been established 
by P.-L. Lions in \cite{L2} (see also \cite{I2}) 
for the Neumann/oblique derivative boundary condition and 
by I.~Capuzzo-Dolcetta and P.-L.~Lions in \cite{CL} 
(see also \cite{M1}) 
for the state constraint boundary condition. 
Additive eigenvalue problems (E) and (D)$_{a}$ give 
the ``stationary states'' for solutions of (IB) 
as our main result shows.  
A simple observation related to this is that, for any 
$(v,c)\in C(\cO)$, 
the function $v(x)-ct$ is a solution of (IB) 
if and only if $(v,c)$ is a solution of 
(E) or (D)$_{c}$. 

As we suggested above, our main purpose is to prove that the solution $u$ of (IB) has, for the Neumann/oblique derivative or the state constraint boundary condition, the following behavior
\begin{equation}\label{conv}
u(x,t)-(v(x)-ct)\to0 \quad\textrm{uniformly on} 
\ \cO 
\ \textrm{as} \ t\to\infty, 
\end{equation}
where $(v,c)\in \W(\Om)\times\R$ is a solution 
of (E-N) or (E-SC) while, for the Dirichlet boundary condition case, this behavior depends on the sign of $c_{sc}$. 
We call such a function $v(x)-ct$ 
an \textit{asymptotic solution} of (IB). 
It is worth mentioning that 
though we can relatively easily get the convergence 
\[
\frac{u(x,t)}{t}\to-c \ \textrm{uniformly on} \ 
\cO \ \textrm{as} \ t\to\infty
\]
as a result of Theorem \ref{thm:additive} and the comparison 
principle for (IB) (see Theorem \ref{thm:comparison}), 
it is generally rather difficult to show 
\eqref{conv}, since (as we already mention it above) additive eigenfunctions 
of (E) or solutions of (D)$_{c}$ may not be unique. 

Before stating our main result on the asymptotic behavior of solution of (E), we recall that, in the last decade, the large time behavior of solutions of Hamilton-Jacobi equation in compact manifold (or in $\R^N$, mainly in the periodic case) has received much attention and general convergence results for solutions have been established. 
G. Namah and J.-M. Roquejoffre in \cite{NR} are the first to 
prove \eqref{conv} 
under the following additional assumption 
\begin{equation}\label{assumption-NR}
H(x,p)\ge H(x,0) 
\ \textrm{for all} \ (x,p)\in\cM\times\R^{N} 
\ \textrm{and} \ 
\max_{\cM}H(x,0)=0,   
\end{equation}
where $\cM$ is a smooth compact $N$-dimensional 
manifold without boundary. 
Then A. Fathi in \cite{F2} proved 
the same type of convergence result 
by dynamical systems type arguments introducing the ``weak KAM 
theory''. Contrarily to \cite{NR}, the results of \cite{F2} use 
strict convexity (and smoothness) assumptions on $H(x,\cdot)$, 
i.e., $D_{pp}H(x,p)\ge\al I$ for all $(x,p)\in\cM\times\R^{N}$ 
and $\al>0$ (and also far more regularity) but do not require
\eqref{assumption-NR}. 
Afterwards J.-M. Roquejoffre \cite{R} and 
A. Davini and A. Siconolfi in \cite{DS} refined 
the approach of A. Fathi and 
they studied the asymptotic problem for 
\eqref{ib-1} on $\cM$ or $N$-dimensional torus. 
We also refer to the articles 
\cite{BR, I, II1, II2, II3} 
for the asymptotic problems in the whole domain $\R^N$ 
without the periodic assumptions in various situations.  
The first author and P. E. Souganidis obtained in \cite{BS} more general results, for possibly non-convex Hamiltonians, by using an approach based on partial differential equations methods and viscosity solutions, which was not using in a crucial way the explicit formulas of representation of
the solutions. Later, the first author and J.-M. Roquejoffre provided results for unbounded solutions, for convex and non-convex equations, using partially the ideas of \cite{BS} but also of \cite{NR}. In this paper, we follow the approach of \cite{BS}.

There also exists results on the asymptotic behavior of solutions of convex Hamilton-Jacobi Equation with boundary conditions. 
The second author \cite{M1} studied
the case of the state constraint boundary condition and then the Dirichlet boundary conditions \cite{M2, M3}. Roquejoffre in \cite{R} was also dealing with solutions of the Cauchy-Dirichlet problem  which satisfy the Dirichlet boundary condition pointwise (in the classical sense)~: this is a key difference with the results of \cite{M2, M3} where the solutions were satisfying the Dirichlet boundary condition in a generalized (viscosity solutions) sense. For Cauchy-Neumann problems, H. Ishii in \cite{I3} establishes very recently the asymptotic behavior. 
We point out that all these works use a generalized dynamical approach similar to \cite{F2, DS}, 
which is very different from ours. Our results are more general since they also hold also for non-convex Hamiltonians. It is worthwhile to mention that our results on 
\eqref{conv} include those in \cite{M1, M2, I3}.

We also refer to the articles \cite{R, BeR} for the large time behavior of 
solutions to time-dependent Hamilton-Jacobi equations and 
\cite{FU, FL} for the convergence rate in \eqref{conv}. 
Recently the second author with Y. Giga and Q. Liu in \cite{GLM1, GLM2} 
has gotten the large time behavior of solutions of 
Hamilton-Jacobi equations with noncoercive Hamiltonian 
which is motivated by a model describing growing faceted crystals. 

To state our main result, we need the following 
assumptions on the Hamiltonian. 
We denote by $H_{a}(x,p)$ the Hamiltonian $H(x,p)-a$ 
for $a\in\R$. 
\begin{itemize}
\item[{\rm(A2)}$_{a}^{+}$] 
There exists $\eta_{0}>0$ such that, 
for any $\eta\in(0,\eta_{0}]$, 
there exists $\psi_{\eta}>0$ such that 
if $H_{a}(x,p+q)\ge\eta$ and $H_{a}(x,q)\le0$ for some 
$x\in\cO$ and $p,q\in\R^{N}$, then for any $\mu\in(0,1]$, 
\[
\mu H_{a}(x,\frac{p}{\mu}+q)\ge 
H_{a}(x,p+q)+\psi_{\eta}(1-\mu). 
\]

\item[{\rm(A2)}$_{a}^{-}$] 
There exists $\eta_{0}>0$ such that, 
for any $\eta\in(0,\eta_{0}]$,
there exists $\psi_{\eta}>0$ such that 
if $H_{a}(x,p+q)\le-\eta$ and $H_{a}(x,q)\ge0$ for some 
$x\in\cO$ and $p,q\in\R^{N}$, then for any $\mu\ge1$, 
\[
\mu H_{a}(x,\frac{p}{\mu}+q)\le 
H_{a}(x,p+q)-\frac{\psi_{\eta}(\mu-1)}{\mu}. 
\]
\end{itemize}

Our main result is the following theorem.
\begin{thm}[Large-Time Asymptotics]\label{thm:large-time}
Assume that {\rm (BA)} holds.\\
{\rm (i)} {\bf (Neumann/oblique derivative problem)} 
If {\rm(A2)}$_{c_{n}}^{+}$ 
or {\rm(A2)}$_{c_{n}}^{-}$ holds, then there exists a solution 
$v \in W^{1,\infty}(\Om)$ of {\rm (E-N)} with $c=c_{n}$ such that 
\eqref{conv} holds with $c=c_{n}$.\\ 
{\rm (ii)} {\bf (State constraint problem)} 
If {\rm(A2)}$_{c_{sc}}^{+}$ or {\rm(A2)}$_{c_{sc}}^{-}$ holds, then there exists a solution $v\in W^{1,\infty}(\Om)$ of {\rm (E-SC)} with $c=c_{sc}$ such that  \eqref{conv} holds with $c=c_{sc}$. \\
{\rm (iii)} {\bf (Dirichlet problem)} 
Assume that {\rm(A2)}$_{c_{sc}}^{+}$ or {\rm(A2)}$_{c_{sc}}^{-}$ holds. \\ 
{\rm (iii-a)} 
If $c_{sc}>0$, 
then there exists a solution $v\in W^{1,\infty}(\Om)$ of {\rm (E-SC)} with $c=c_{sc}$ such that  \eqref{conv} holds with $c=c_{sc}$. \\
{\rm (iii-b)} 
If $c_{sc} \le 0$, then there exists a solution $v\in W^{1,\infty}(\Om)$ of {\rm(D)$_{0}$} 
such that \eqref{conv} holds with $c=0$.
\end{thm}

Assumption (A2)$^{+}_{a}$ is introduced in \cite{BS} to replace the convexity assumption : it mainly concerns the set $\{H_a \geq 0\}$ and the behavior of $H_a$ in this set. 
Assumption (A2)$^{-}_{a}$ is a modification of (A2)$^{+}_{a}$ which, on the contrary, concerns the set $\{H_a \leq 0\}$. 
We can generalize them as in \cite{BS} (see Remark \ref{rem:general} (ii)) 
but to simplify our arguments we only use the simplified version 
in Theorem \ref{thm:large-time}.  
It is worthwhile to mention that we can deal with 
$u_{0}\in C(\cO)$ and prove that uniform continuous 
solutions of (IB) (instead of Lipschitz continuous solutions) 
have the same large-time asymptotics given by 
Theorem \ref{thm:large-time}. 
See Remark \ref{rem:general} (i).  
We also mention that assumptions 
(A2)$^{+}_{c}$ and (A2)$^{-}_{c}$ are 
equivalent as (A7)$_{+}$ and (A7)$_{-}$ in \cite{II3} 
under the assumption that $H$ is coercive and convex 
with respect to $p$-variable, 
where $c$ is the associated constant as in Theorem \ref{thm:large-time} 
(see \cite[Appencix C]{II1}).

This paper is organized as follows: 
in Section 2 we recall the proof of Theorem \ref{thm:additive}. 
In Section 3 we present \textit{asymptotically monotone property}, 
which is a key tool to prove Theorem \ref{thm:large-time}. 
Section 4 is devoted to the proof of a key lemma 
to prove asymptotically monotone properties. 
In Section 5, we give the proof of Theorem \ref{thm:large-time} 
and we give some remarks on the large-time behavior of solutions 
of convex Hamilton-Jacobi equations in Section 6. 
In Appendix we present the definition, 
existence, uniqueness and regularity results for {\rm (IB)}.

\medskip
\noindent
{\bf Notations.} 
Let $\R^{N}$ denote 
the $N$-dimensional Euclidean space for some $k\in\N$. 
We denote by $|\cdot|$ the usual Euclidean norm. 
We write $B(x,r)=\{y\in\R^{N}\mid |x-y|<r\}$ for $x\in\R^{N},$ 
$r>0$. 
For $A\subset\R^{N}$, 
we denote by 
$C(A)$, $\BUC(A)$, 
$\LSC(A)$, $\USC(A)$  and $C^k(A)$ 
the space of real-valued 
continuous, 
bounded uniformly continuous, 
lower semicontinuous, 
upper semicontinuous  
and $k$-th continuous differentiable functions  
on $A$ for $k\in\N$, respectively.  
We write 
$a\land b=\min\{a,b\}$ and 
$a\vee b=\max\{a,b\}$ for $a,b\in\R$. 
We call a function $m:[0,\infty)\to[0,\infty)$ 
a modulus if it is continuous and nondecreasing 
on $[0,\infty)$ and vanishes at the origin. 
For any $k\in\N\cup\{0\}$ 
we call $\Om$ a domain with a $C^{k}$ boundary if 
$\Om$ is has a boundary $\bO$ locally represented 
as the graph of a $C^{k}$ function, 
i.e., 
for each $z\in\bO$ there exist $r>0$ and 
a function $b\in C^{k}(\R^{N-1})$ such that 
---upon relabelling and re-orienting the coordinates axes 
if necessary---
we have 
\[
\Om\cap B(z,r)=\{(x',x_{N})\in \R^{N-1}\times\R\mid 
x\in B(z,r),  x_{N}>b(x')\}.
\]


\section{Additive Eigenvalue Problem}

\begin{proof}[Proof of Theorem {\rm \ref{thm:additive}}] We just sketch the proof since it is an easy adaptation of classical arguments.

We start with (E-N) and for any $\ep\in (0,1)$, we consider 
\begin{equation}\label{pf:additive1}
\left\{
\begin{aligned}
&
\ep \ue +H(x,D\ue )=0
&& \textrm{in} \ \Om, \\
&B(x,\ue ,D\ue )=0
&& \textrm{on} \ \bO. 
\end{aligned}
\right.
\end{equation}
In order to apply Perron's method, we first build a subsolution 
and a supersolution.

We claim that $\u1e^{-} = -M_1/\ep$ and $\u1e^{+} = M_1/\ep$ 
are respectively sub and supersolution of \eqref{pf:additive1}. 
This is obvious in $\Om$ for a suitable large $M_{1}>0$ since $H(x,0)$ is continuous, hence bounded, on $\cO$. 
If $x \in \bO$, we recall (see \cite{ L2, CIL} for instance) that the super-differential of $\u1e^{-}$ at $x$ consists in elements of the form $\lambda n(x)$ with $\lambda \leq 0$ and we have to show that 
\[
\min\{-M_1 + H(x,\lambda n(x)),\lambda n(x)\cdot\gamma(x) -g(x)\} \leq 0. 
\]
By \eqref{oblique} it is clear enough that there exists $\bar\lambda < 0$ such that, if $\lambda \leq \bar\lambda$, then 
$\lambda n(x)\cdot\gamma(x) -g(x) \leq 0$. Then choosing $M_1 \geq \max\{H(x,\lambda n(x))\mid  x\in \bO,\  \bar\lambda \leq \lambda \leq 0\}$, the above inequality holds. A similar argument shows that $\u1e^{+}$ is a supersolution of \eqref{pf:additive1} for $M_1$ large enough.
 
We remark that, because of (A1) and the regularity of the boundary of $\Om$, the subsolutions $w$ of \eqref{pf:additive1} such that $\u1e^{-}  \leq w \leq \u1e^{+}$ on $\cO$ satisfy
$|Dw|\leq M_2$ in $\Om$ and therefore they are equi-Lipschitz continuous on $\cO$. With these informations, Perron's method provides us with a solution $u_{\ep}\in W^{1,\infty}(\Om)$ of \eqref{pf:additive1}. 
Moreover, by construction, we have
\begin{equation}\label{pf:additive2} 
|\ep u_{\ep}|\le M_{1} \ \textrm{on} \ \cO \quad \hbox{and}\quad
|Du_{\ep}|\le M_{2} \ \textrm{in} \ \Om.
\end{equation}

Next we set $v_{\ep}(x):=u_{\ep}(x)-u_{\ep}(x_{0})$ for a fixed $x_{0}\in\cO$. 
Because of \eqref{pf:additive2} and the regularity of the boundary $\bO$, 
$\{v_{\ep}\}_{\ep\in(0,1)}$ is a sequence of equi-Lipschitz continuous and uniformly bounded functions
on $\cO$.  
By Ascoli-Arzela's Theorem, 
there exist subsequences
$\{v_{\ep_{j}}\}_{j}$ 
and 
$\{u_{\ep_{j}}\}_{j}$ 
such that 
\[
v_{\ep_{j}}\to v, \;
\ep_{j}u_{\ep_{j}}\to -c_{n} \ 
\textrm{uniformly on} \ \cO
\]
as $j\to\infty$ 
for some $v\in W^{1,\infty}(\Om)$ and $c_{n}\in\R$. 
By a standard stability result of viscosity solutions 
we see that $(v,c_{n})$ is a solution of (E-N).

We prove the uniqueness of additive eigenvalues. 
Let $(v_{a},a)\in W^{1,\infty}(\Om)\times\R$ be a subsolution of (E-N). 
Adding a positive constant to $v$ if necessary we may 
assume that $v_{a}\le v$ on $\cO$. 
Note that $v-c_{n}t$ and $v_{a}-at$ 
are, respectively, a solution and a subsolution 
of \eqref{ib-1} and \eqref{ib-2}. 
By the comparison principle for (IB) 
(see Theorem \ref{thm:comparison}), we have 
\[
v_{a}-at\le v-c_{n}t \ \textrm{on} \ \cQ. 
\]
Therefore we have 
\begin{equation}\label{pf:additive3}
v_{a}-v\le (a-c_{n})t \ \textrm{on} \ \cQ.
\end{equation}
If $a<c_{n}$, then \eqref{pf:additive3} implies 
a contradiction for a large $t>0$ since $(a-c_{n})t \to -\infty$ as $t\to +\infty$. 
Therefore we obtain $c_{n}\le a$ and thus 
we see that $c$ can be represented by 
\[
c_{n}=\inf\{a\in\R\mid 
{\rm (E}\textrm{-}{\rm N)} 
\ \textrm{has a subsolution}\}. 
\]
The uniqueness of additive eigenvalues follows immediately.

Next we consider (E-SC). The proof follows along the same lines, except for the following point. \\
(i) 
For a subsolution, it is enough to choose the constant $-m/\ep$, 
where $m= \max_{x\in \cO}H(x,0)$ since only the inequality in $\Om$ is needed. 
For a supersolution, we choose $M/\ep$, where 
$M:=-\min_{\cO\times\R^{N}}H(x,p)$. \\
(ii) Since the boundary has only $C^0$-regularity, 
the solutions are in $W^{1,\infty}(\Om)$ but may not be Lipschitz continuous up to the boundary. But they are equi-continuous (see \cite[Proposition 8.1]{IM}) 
and this is enough to complete the argument.\\
(iii) We have
$$c_{sc} = \inf\{a\in\R\mid \eqref{erg-1} \ \textrm{has a subsolution}\},$$ 
by the same argument as above since, in the state-constraint case, there are no boundary conditions for the subsolutions. 

Finally we consider (D)$_{a}$. 
In the case where $c_{sc}>a$, 
it is clear that 
(D)$_{a}$ has no solutions due to the definition 
of $c_{sc}$. 
Thus we need only to show that 
(D)$_{a}$ has a solution in the case where $c_{sc}\le a$.

Due to the coercivity of $H$, 
there exists $p_0\in\R^{N}$ such that $H(x,p_0)\ge a$ 
for all $x\in\overline\Omega$. 
We set $\psi_1(x)=p_0\cdot x+C$, 
where $C>0$ is a constant chosen 
so that $\psi_1(x)\ge g(x)$ for all $x\in\partial\Om$.  
Since $a \geq c_{sc}$, any solution $\psi_2\in C(\ol\Omega)$
of (E-SC) is a subsolution of \eqref{d-1} and subtracting a sufficiently large constant from $\psi_2$ if necessary, 
we may assume that $\psi_2(x)\le g(x)$ on $\partial\Om$ and $\psi_2\le\psi_1$ on $\ol\Om$.
As a consequence of Perron's method 
we see that there exists a solution $v$ of (D)$_{a}$ 
for any $a\ge c_{sc}$. 
\end{proof}

The following proposition shows that, taking into account the ergodic effect, we obtain bounded solutions of (IB). This result is a straightforward consequence of 
Theorems \ref{thm:additive} and \ref{thm:comparison}.
\begin{prop}[Boundedness of Solutions of {\rm (IB)}]\label{prop:bound}
Assume that {\rm (BA)} holds. Let $c_{n}$ and $c_{sc}$ be the constants given by 
\eqref{additive-const-neumann} and \eqref{additive-const-state}, 
respectively. \\
{\rm (i)} 
Let $u$ be the solution of {\rm (CN)}. 
Then $u+c_{n}t$ is bounded on $\cQ$. \\
{\rm (ii)} 
Let $u$ be the solution of {\rm (SC)}. 
Then $u+c_{sc}t$ is bounded on $\cQ$. \\
{\rm (iii)} 
Let $u$ be the solution of {\rm (CD)}. \\
{\rm (iii-a)} 
If $c_{sc}>0$, 
then $u+c_{sc}t$ is bounded on $\cQ$. \\
{\rm (iii-b)} 
If $c_{sc}\le0$, 
then $u$ is bounded on $\cQ$. 
\end{prop}

\section{Asymptotically Monotone Property}

Following Proposition~\ref{prop:bound}, it turns out that, in order to go further in the study of the asymptotic behavior of the solutions $u$ of (IB), we have to consider $u_{c}:=u+ct$, where $c$ is the associated additive eigenvalue: $c:= c_{n}$ for (CN), $c:=c_{sc}$ for (SC) and $c:=c_{sc}$ if $c_{sc}>0$ and $c:=0$ if $c_{sc}\le0$ for (CD). A first step to prove the convergence of the $u_c$'s is the following result on their asymptotic monotonicity in time which can be stated exactly in the same way for the three cases, namely (CN), (SC) and (CD).

\begin{thm}[Asymptotically Monotone Property]\label{thm:asymp-mono} 
Assume that {\rm (BA)} holds. \\
{\rm (i)} 
{\bf (Asymptotically Increasing Property)} \\
Assume that {\rm(A2)}$_{c}^{+}$ holds. 
For any $\eta\in(0,\eta_{0}]$, there exists 
$\del_{\eta}: [0,\infty)\to[0,1]$ such that 
\begin{align*}
&\del_{\eta}(s)\to0 \quad\textrm{as} \ s\to\infty \quad 
\textrm{and}\\
&u_{c}(x,s)-u_{c}(x,t)+\eta(s-t)\le \del_{\eta}(s)
\end{align*}
for all $x\in\cO$, $s,t\in[0,\infty)$ with $t\ge s$. \\
{\rm (ii)} {\bf (Asymptotically Decreasing Property)} \\
Assume that {\rm(A2)}$_{c}^{-}$ holds. 
For any $\eta\in(0,\eta_{0}]$, there exists 
$\del_{\eta}: [0,\infty)\to[0,\infty)$ such that 
\begin{align*}
&\del_{\eta}(s)\to0 \quad\textrm{as} \ s\to\infty \quad 
\textrm{and}\\
&u_{c}(x,t)-u_{c}(x,s)-\eta(t-s)\le \del_{\eta}(s)
\end{align*}
for all $x\in\cO$, $s,t\in[0,\infty)$ with $t\ge s$. 
\end{thm}

In the same way as for $u_c$, we denote by $v_{c}$ a subsolution of (E) which means, of course, either (E-N) or (E-SC). We keep the subscript ``$c$'' to remind that it is associated to either $c_{n}$ or $c_{sc}$. 

By Proposition \ref{prop:bound} $u_{c}$ is bounded on $\cQ$. 
Note that $v_{c}-M$ are still subsolutions of (E) or (D)$_{c}$ for 
any $M>0$. 
Therefore subtracting a positive constant to $v_{c}$ if necessary, 
we may assume that 
\begin{equation}\label{bdd:u-v}
1\le u_{c}(x,t)-v_{c}(x)\le C 
\quad\textrm{for all} \ (x,t)\in\cQ \ \textrm{and some} \ 
C>0 
\end{equation}
and we fix such a constant $C$. 
We define the functions
$\mu_{\eta}^{+}, \mu_{\eta}^{-}:[0,\infty)\to\R$ by 
\begin{align}
&
\mu_{\eta}^{+}(s):=\min_{x\in\cO, t\ge s}
\Bigr(\frac{u_{c}(x,t)-v_{c}(x)+\eta(t-s)}{u_{c}(x,s)-v_{c}(x)}\Bigr), 
\label{def:mu-plus} \\
&
\mu_{\eta}^{-}(s):=\max_{x\in\cO, t\ge s}
\Bigr(\frac{u_{c}(x,t)-v_{c}(x)-\eta(t-s)}{u_{c}(x,s)-v_{c}(x)}\Bigr)
\nonumber
\end{align}
for $\eta\in(0,\eta_{0}]$. 
By the uniform continuity of $u_{c}$ and $v_{c}$, 
we have 
$\mu_{\eta}^{\pm}\in C([0,\infty))$.  
It is easily seen that
$0\le\mu_{\eta}^{+}(s)\le 1$ and 
$\mu_{\eta}^{-}(s)\ge 1$ 
for all $s\in[0,\infty)$ 
and $\eta\in(0,\eta_{0}]$.

\begin{prop}\label{prop:mu-to-1}
Assume that {\rm (BA)} holds.\\ 
{\rm (i)} 
Assume that {\rm(A2)}$_{c}^{+}$ holds. 
We have $\mu_{\eta}^{+}(s)\to 1$ 
as $s\to\infty$ for any $\eta\in(0,\eta_{0}]$. \\
{\rm (ii)} 
Assume that {\rm(A2)}$_{c}^{-}$ holds. 
We have $\mu_{\eta}^{-}(s)\to 1$ 
as $s\to\infty$ for any $\eta\in(0,\eta_{0}]$. \\
\end{prop}

Proposition~\ref{prop:mu-to-1} is a consequence of the following lemma. 

\begin{lem}[Key Lemma]\label{lem:mu}
Assume that {\rm (BA)} holds.
Let $C$ be the constant given by \eqref{bdd:u-v}. \\
{\rm (i)} 
Assume that {\rm(A2)}$_{c}^{+}$ holds. 
The function $\mu_{\eta}^{+}$ is a supersolution of  
\begin{equation}\label{variational-ineq}
\max\{w(s)-1, 
w^{'}(s)+\frac{{\psi_{\eta}}}{C}(w(s)-1)\}
=0 \ \textrm{in} \ (0,\infty) 
\end{equation}
for any $\eta\in(0,\eta_{0}]$. \\
{\rm (ii)} 
Assume that {\rm(A2)}$_{c}^{-}$ holds. 
The function $\mu_{\eta}^{-}$ is a subsolution of  
\begin{equation}\label{variational-ineq2}
\max\{w(s)-1, 
w^{'}(s)
+\frac{\psi_{\eta}}{C}\cdot\frac{w(s)-1}{w(s)}\}
=0 \ \textrm{in} \ (0,\infty) 
\end{equation}
for any $\eta\in(0,\eta_{0}]$. 
\end{lem}
We first prove Proposition \ref{prop:mu-to-1} 
and Theorem \ref{thm:asymp-mono}
by using Lemma \ref{lem:mu}  
and we postpone to 
prove Lemma \ref{lem:mu} in the next section.

\begin{proof}[Proof of Proposition {\rm \ref{prop:mu-to-1}}]
Fix $\eta\in(0,\eta_{0}]$. 
We first notice that, by definition, 
$$\mu_{\eta}^{+}(s) \leq 1 \leq \mu_{\eta}^{-}(s)\; ,$$
for any $s\geq 0$. On the other hand, one checks easily that the functions
$$ 1 +  (\mu_{\eta}^{+}(0)-1) \exp(-\frac{\psi_{\eta}}{C}t)\; ,$$
and
$$ 1 +  (\mu_{\eta}^{-}(0)-1) \exp(-\frac{\psi_{\eta}}{C\mu_{\eta}^{-}(0)}t)\; ,$$
are, respectively, sub and supersolution of (\ref{variational-ineq}) or (\ref{variational-ineq2}) and they take the same value as $\mu_{\eta}^{+}$ and $\mu_{\eta}^{-}$ for $s=0$, respectively. Therefore, by comparison, we have for any $s>0$
$$ 1 +  (\mu_{\eta}^{+}(0)-1) \exp(-\frac{\psi_{\eta}}{C}t) \leq \mu_{\eta}^{+}(s)\; ,$$
and
$$\mu_{\eta}^{-}(s) \leq 1 +  (\mu_{\eta}^{-}(0)-1) \exp(-\frac{\psi_{\eta}}{C\mu_{\eta}^{-}(0)}t)\; ,$$
which imply the conclusion. 
\end{proof}

\begin{proof}[Proof of Theorem {\rm \ref{thm:asymp-mono}}]
We only prove (i), since we can prove (ii) similarly.  
For any $x\in\cO$ and $t,s\in[0,\infty)$ with $t\ge s$, 
we have 
\[
\mu_{\eta}^{+}(s)(u_{c}(x,s)-v_{c}(x))\le 
u_{c}(x,t)-v_{c}(x)+\eta(t-s),
\]
which implies 
\begin{align*}
u_{c}(x,s)-u_{c}(x,t)+\eta(s-t)
&\le\, (1-\mu_{\eta}^{+}(s))(u_{c}(x,s)-v_{c}(x))\\
{}&\le\, C(1-\mu_{\eta}^{+}(s))
=: \del_{\eta}(s). 
\end{align*}
Therefore we have the conclusion. 
\end{proof}

\section{Proof of Lemma {\rm \ref{lem:mu}}}
In this section 
we divide into three subsections and 
we prove Lemma {\rm \ref{lem:mu}} for 
three types boundary problems. 
We only prove (i), since we can prove (ii) similarly 
in any case.

\subsection{Neumann Problem}
In this subsection we normalize the additive eigenvalue $c_{n}$ to be $0$ 
by replacing $H$ by $H-c_{n}$, where $c_{n}$ is the constant given by 
\eqref{additive-const-neumann}.

\begin{proof}[Proof of Lemma {\rm \ref{lem:mu}}]
Fix $\mu\in(0,\eta_{0}]$ 
and let $\mu_{\eta}^{+}$ be the function given 
by \eqref{def:mu-plus}. 
By abuse of notation we write $\mu$ 
for $\mu_{\eta}^{+}$. 
We recall that $\mu(s)\leq 1$ for any $s\geq 0$.
 
Let $\phi\in C^{1}((0,\infty))$ and $\sig>0$ be a strict 
local minimum of $\mu-\phi$, 
i.e., 
$(\mu-\phi)(s)>(\mu-\phi)(\sig)$ 
for all $s\in[\sig-\del,\sig+\del]\setminus\{\sig\}$ and some $\del>0$.  
Since there is nothing to check $\mu(\sig)=1$, 
we assume that $\mu(\sig)<1$. 
We choose $\xi\in\cO$ and $\tau\ge\sig$ 
such that 
\[
\mu(\sig)=
\frac{u(\xi,\tau)-v(\xi)+\eta(\tau-\sig)}
{u(\xi,\sig)-v(\xi)}. 
\]

Since we can get the conclusion by the same argument 
as in \cite{BS} in the case where $\xi\in\Om$, 
we only consider the case where $\xi\in\bO$ in this proof.

Next, for $\al >0$ small enough, we consider the function 
\[
(x,t,s)\mapsto 
\frac{u(x,t)-v(x)+\eta(t-s)}
{u(x,s)-v(x)}+|x-\xi|^{2} + |t-\tau|^{2} -\phi(s) -3\al d(x)\; .
\]
We notice that, for $\al=0$, $(\xi,\tau, \sig)$ is a strict minimum point of this function. This implies that, for $\al>0$ small enough, this function achieves its minimum over $\cO\times\{(t,s)\mid t\ge s,\ s\in[\sig-\del,\sig+\del]\}$
at some point $(\xi_{\al}, t_{\al}, s_{\al})$ which converges to $(\xi,\tau, \sig)$ when $\al \to 0$. 
Then there are two cases : either (i) $\xi_{\al}\in\Om$ or (ii) $\xi_{\al}\in\bO$. We only consider case (ii) here too since, again, the conclusion follows by the same argument as in \cite{BS} in case (i).

In case (ii), since $d(\xi_{\al})=0$, the $\al$-term vanishes and we have $(\xi_{\al}, t_{\al}, s_{\al}) =(\xi, \tau, \sig)$ by the strict minimum point property.

Choose a positive definite symmetric matrix $A$  
which satisfies 
$A\gam(\xi)=n(\xi)$ (see \cite[Lemma 4.1]{LS}) and 
define the quadratic function $f:\cO\to\R$ by 
\begin{equation}\label{func-f}
f(x):=\frac{1}{2}Ax\cdot x. 
\end{equation}
For $\ep\in(0,1)$,
we set $K:=\cO^{3}\times\{(t,s)\mid t\ge s, 
s\in[\sig-\del,\sig+\del]\}$ and
define the function 
$\Psi:K\to\R$ by 
\begin{align*}
{}&\Psi(x,y,z,t,s)\\
:=&
\frac{u(x,t)-v(z)+\eta(t-s)
+g(\xi)(d(x)-d(z))}{u(y,s)-v(z)
+g(\xi)(d(y)-d(z))}-\phi(s)\\
{}&
+\frac{1}{\ep^{2}}(f(x-y)+f(x-z))+|x-\xi|^{2} + |t-\tau|^{2}
-\al(d(x)+d(y)+d(z)), 
\end{align*}
where 
\begin{equation}\label{func-d}
d(x):=-\frac{\gam(\xi)\cdot x}{|\gam(\xi)|^{2}}. 
\end{equation}
Let $\Psi$ achieve its minimum over 
$K$ at some $(\ol{x}, \ol{y}, \ol{z}, \ol{t}, \ol{s})$. 
Set 
\begin{align*}
\ol{\mu}_{1}&:=
u(\ol{y},\ol{s})-v(\ol{z})+g(\xi)
(d(\ol{y})-d(\ol{z})), \\
\ol{\mu}_{2}&:=
u(\ol{x},\ol{t})-v(\ol{z})+\eta(\ol{t}-\ol{s})+g(\xi)
(d(\ol{x})-d(\ol{z})), \\
\ol{\mu}&:=\frac{\ol{\mu}_{2}}{\ol{\mu}_{1}}. 
\end{align*}

Since $\cO$ is compact, by taking a subsequence if necessary, 
we may assume that 
\[\ol{x}, \ol{y}, \ol{z}\to \xi\ 
\textrm{and} \ 
\ol{t} \to \tau, \ 
\ol{s}\to\sig \ 
\textrm{as} \ \ep\to0. 
\]
Moreover, we have 
\begin{gather}
\ol{\mu}\to\mu(\sig) 
\ \textrm{as} \ \ep\to0, \label{pf-main:c1}\\
\frac{|\ol{x}-\ol{y}|}{\ep^{2}}
+\frac{|\ol{x}-\ol{z}|}{\ep^{2}}\le C_{1} 
\ \textrm{for some} \ C_{1}>0. \label{pf-main:bdd-lip}
\end{gather}
The second property is due to the Lipschitz continuity of 
$u$ and $v$. 
From \eqref{pf-main:c1}, we may assume that 
$\ol{\mu}<1$ for small $\ep>0$.

Derivating $\Psi$ with respect to each variable $x,t,y,s,z$ 
at $(\ol{x},\ol{y},\ol{z},\ol{t},\ol{s})$ formally,  we get 
\begin{align*}
D_{x}u(\ol{x},\ol{t})
&=\, 
\ol{\mu}_{1}\bigl(
\frac{1}{\ep^{2}}A(\ol{y}-\ol{x})+
\frac{1}{\ep^{2}}A(\ol{z}-\ol{x})
+2(\xi-\ol{x})+\al Dd(\ol{x})
\bigr)\\
{}&\hspace*{13pt}-g(\xi)Dd(\ol{x}),\\
u_{t}(\ol{x},\ol{t})&=\, -\eta -2 \ol{\mu}_{1}(\ol{t}-\tau), \\
D_{y}u(\ol{y},\ol{s})&=\, 
\frac{\ol{\mu}_{1}}{\ol{\mu}}\bigl(
\frac{1}{\ep^{2}}A(\ol{y}-\ol{x})-\al Dd(\ol{y})
\bigr)-g(\xi)Dd(\ol{y})\\
u_{s}(\ol{y},\ol{s})&=\, 
-\frac{1}{\ol{\mu}}(\eta+\ol{\mu}_{1}\phi^{'}(\ol{s})), \\
D_{z}v(\ol{z})&=\,
\frac{\ol{\mu}_{1}}{1-\ol{\mu}}\bigl(
\frac{1}{\ep^{2}}A(\ol{z}-\ol{x})-\al Dd(\ol{z})\bigr)
-g(\xi)Dd(\ol{z}). 
\end{align*}
We remark that 
we should interpret $D_{x}u, u_{t}, D_{y}u, u_{s}$ and $D_{z}v$ as 
the viscosity solution sense here.

We consider the case where $\ol{x}\in\bO$ and 
then we obtain 
\[
D_{x}u(\ol{x},\ol{t})\cdot \gam(\ol{x})
\le g(\ol{x})+m(\ep)-\al n(\xi)\cdot\gam(\xi)
\]
for some modulus of continuity $m$ by similar arguments in the proof of 
Theorem \ref{thm:comparison}.   
Therefore, if $\ep>0$ is suitable small compared to $\al>0$, 
we have 
\[
D_{x}u(\ol{x},\ol{t})\cdot \gam(\ol{x})<g(\ol{x}). 
\]
Similarly, if $\ol{y}\in\bO$ and $\ol{z}\in\bO$, then 
we have 
\begin{align*}
&D_{y}u(\ol{y},\ol{s})\cdot \gam(\ol{y})>g(\ol{y}) \quad \textrm{and}\\
&D_{z}v(\ol{z})\cdot \gam(\ol{z})>g(\ol{z}), 
\end{align*}
respectively, 
for $\ep>0$ which is small enough compared to $\al>0$.

Therefore, by the definition of viscosity solutions of (CN), 
we have 
\begin{align}
-\eta -2 \ol{\mu}_{1}(\ol{t}-\tau) +H(\ol{x}, D_{x}u(\ol{x},\ol{t}))\ge0, \label{pf-main:ineq1}\\
-\frac{1}{\ol{\mu}}(\eta+\ol{\mu}_{1}\phi^{'}(\ol{s}))+
H(\ol{y}, D_{y}u(\ol{y},\ol{s}))\le0, \label{pf-main:ineq2}\\
H(\ol{z}, D_{z}v(\ol{z}))\le0. \label{pf-main:ineq3}
\end{align}
In view of \eqref{pf-main:bdd-lip} 
we may assume that 
\[
\frac{1}{\ep^{2}}A(\ol{y}-\ol{x})\to p_{y}, \quad 
\frac{1}{\ep^{2}}A(\ol{z}-\ol{x})\to p_{z}
\]
as $\ep\to0$ for some $p_y, p_z\in\R^{N}$ 
by taking a subsequence if necessary. 
Set 
\begin{align*}
&\mu_{1}:=u(\xi,\sig)-v(\xi), \\
&P:=\frac{\mu_{1}}{\mu(\sig)}p_{y}, \\
&Q:=\frac{\mu_{1}}{1-\mu(\sig)}p_{z}, \\
&\tilde{P}:=\mu(\sig)(P-Q) \ 
\textrm{and}\\
&\tilde{Q}:=Q-g(\xi)Dd(\xi). 
\end{align*}
Sending $\ep\to0$ and then $\al\to0$ 
in \eqref{pf-main:ineq1} and \eqref{pf-main:ineq3} and recalling that $\ol{t} \to \tau$,
we obtain 
\begin{equation}\label{pf-main:ineq4}
H(\xi, \tilde{P}+\tilde{Q})\ge\eta  \ \textrm{and}\ \  H(\xi, \tilde{Q})\le0.
\end{equation}
Therefore, since we have 
$\tilde{P}/\mu(\sig)+\tilde{Q}=P-g(\xi)Dd(\xi)$, 
by using (A2)$_{0}^{+}$, we obtain 
\begin{equation}\label{pf-main:ineq5}
H(\xi, \tilde{P}+\tilde{Q})
\le\, 
\mu(\sig)
H(\xi, P-g(\xi)Dd(\xi))
-\psi_{\eta}(1-\mu(\sig)) 
\end{equation}
for some $\psi_{\eta}>0$. 

Sending $\ep\to0$ and then $\al\to0$ 
in \eqref{pf-main:ineq2}, 
we have 
\begin{equation}\label{pf-main:ineq6}
-\frac{1}{\mu(\sig)}
(\eta+\mu_{1}\phi^{'}(\sig))
+H(\xi, P-g(\xi)Dd(\xi))
\le0. 
\end{equation}
Therefore by \eqref{pf-main:ineq4}, \eqref{pf-main:ineq5} and 
\eqref{pf-main:ineq6} we obtain 
\[
\eta
\le
H(\xi, \tilde{P}+\tilde{Q})
\le
\eta+\mu_{1}\phi^{'}(\sig)+
\psi_{\eta}(\mu(\sig)-1),  
\]
which implies the conclusion. 
\end{proof}

\subsection{State Constraint Problem}
In this subsection 
we normalize the additive eigenvalue $c_{sc}$ is $0$ 
by replacing $H$ by $H-c_{sc}$, 
where $c_{sc}$ is the constant given by 
\eqref{additive-const-state}.

\begin{proof}[Proof of Lemma {\rm \ref{lem:mu}}]
Fix $\eta\in(0,\eta_{0}]$ 
and let $\mu_{\eta}^{+}$ be the function given by \eqref{def:mu-plus}. 
By abuse of notation we write $\mu$  
for $\mu_{\eta}^{+}$. 
Let $\mu-\phi$ take 
a strict local minimum at $\sig>0$ 
and for some $\phi\in C^{1}((0,\infty))$, 
i.e., 
$(\mu-\phi)(t)>(\mu-\phi)(\sig)$ 
for all $t\in[\sig-\del,\sig+\del]\setminus\{\sig\}$ 
and some $\del>0$.  
Since there is nothing to check 
in the case where $\mu(\sig)=1$, 
we assume that $\mu(\sig)<1$. 
We choose $\xi\in\cO$ and $\tau\ge\sig$ 
such that 
\[
\mu(\sig)=
\frac{u(\xi,\tau)-v(\xi)+\eta(\tau-\sig)}
{u(\xi,\sig)-v(\xi)}. 
\]
We only consider the case where $\xi\in\bO$.

Since $\Om$ is a domain with a $C^{0}$-boundary, 
after relabelling and re-orienting the coordinates axes 
if necessary, we may assume that 
$B(\xi,r)\cap \Om
=B(\xi,r)\cap\{(x',x_{N})\in \R^{N-1}\times\R\mid 
x_{N}>b(x')\}$ 
for some constant $r>0$ and 
some continuous function $b$ on $\R^{N-1}$. 
Replacing $r$ by a smaller positive number 
if necessary, 
we may assume that there is an $a_0>0$ such that 
if $0<a\le a_0$, 
then we have $B(\xi,r)\cap\cO\subset -a e_{N}+\Om$, 
where $e_{N}:=(0,...,0,1)\in\R^{N}$.  
Henceforth we assume that $0<a\le a_0$. 
We may choose a bounded open neighborhood $W_{a}$ of 
$\ol{B}(\xi,r)\cap\cO$ such that 
$\ol{W}_{a}\subset-a e_{N}+\Om$.

Set 
\[
u^{a}(x,t):=u(x+a e_{N},t) \ 
\textrm{and} \  
v^{a}(x):=v(x+a e_{N})
\]
for $x\in-a e_{N}+\Omega$ and $t\in[0,\infty)$. 
It is clear to see that 
\begin{gather*}
(u^{a})_{t}+H(x+ae_{N},Du^{a}(x,t))=0 
\ \textrm{in} \ W_{a}\times(\sig-\del,\sig+\del), \\
H(x+ae_{N},Dv^{a}(x,t))=0 
\ \textrm{in} \ W_{a} 
\end{gather*}
in the viscosity sense. 
We set 
\[
\mu_{a}(s):=\min_{x\in\cO, t\ge s}
\Bigr(\frac{u(x,t)-v^{a}(x)+\eta(t-s)}{u^{a}(x,s)-v^{a}(x)}\Bigr) 
\]
and let $\mu_{a}-\phi$ take a minimum 
at some $\sig_{a}\in[\sig-\del,\sig+\del]$ 
and then we have 
$\sig_{a}\to \sig$ and 
$\mu_{a}(\sig_{a})\to\mu(\sig)$ 
as $a\to0$. 
We choose $\xi_{a}\in\cO$ and $\tau_{a}\ge\sig_{a}$ 
such that 
\[
\mu_{a}(\sig)=
\frac{u(\xi_{a},\tau_{a})-v^{a}(\xi_{a})+\eta(\tau_{a}-\sig_{a})}
{u^{a}(\xi_{a},\sig_{a})-v^{a}(\xi_{a})}. 
\]

We set 
$K=(\ol{B}(\xi,r)\cap\cO)\times 
\ol{W}_{a}^{2}\times\{(t,s)\mid t\ge s, s\in[\sig-\del,\sig+\del]\}$ 
and  define the function 
$\Psi:K\to\R$ by 
\begin{align*}
\Psi(x,y,z,t,s)
:=&
\frac{u(x,t)-v^{a}(z)+\eta(t-s)}{u^{a}(y,s)-v^{a}(z)}\\
{}&
+\frac{1}{2\ep^{2}}(|x-y|^{2}+|x-z|^{2})+|x-\xi_{a}|^{2}
+|s-\sig_{a}|^{2}-\phi(s). 
\end{align*}
Let $\Psi$ achieve its minimum on $K$ 
at some $(\ol{x}, \ol{y}, \ol{z}, \ol{t}, \ol{s})$ and 
set 
\begin{align*}
\ol{\mu}_{1a}
&:=u^{a}(\ol{y},\ol{s})-v^{a}(\ol{z}), \\
\ol{\mu}_{2a}
&:=
u(\ol{x},\ol{t})-v^{a}(\ol{z})+
\eta(\ol{t}-\ol{s}), \\
\ol{\mu}_{a}
&:=\frac{\ol{\mu}_{2a}}{\ol{\mu}_{1a}}. 
\end{align*}
Since we may assume by taking a subsequence if necessary that 
\[
\ol{x}, \ol{y}, \ol{z}\to \xi_{a}, 
\ol{s}\to\sig_{a}, 
\ \textrm{and} \ 
\ol{\mu}_{a}\to\mu_{a}(\sig_{a}) 
\ \textrm{as} \ \ep\to0, 
\]
we have 
$\ol{y}, \ol{z}\in W_{a}$ if $\ep$ is small enough.

Therefore, by the definition of viscosity solutions 
we have 
\begin{align}
-\eta+H(\ol{x}, D_{x}u(\ol{x},\ol{t}))\ge0, \label{pf:mu-0}\\
-\frac{1}{\ol{\mu}_{a}}
\bigl(\eta+
\ol{\mu}_{1a}(\phi^{'}(\ol{s})-2(\ol{s}-\sig_{a}))
+H(\ol{y}+a e_{N}, D_{y}u^{a}(\ol{y},\ol{s}))\le0, \label{pf:mu-1}\\
H(\ol{z}+a e_{N}, D_{z}v^{a}(\ol{z}))\le0,  \label{pf:mu-2}
\end{align}
where 
\begin{align*}
D_{x}u(\ol{x},\ol{t})=& \, 
\ol{\mu}_{1a}
\Bigl(
\frac{\ol{y}-\ol{x}}{\ep^{2}}+
\frac{\ol{z}-\ol{x}}{\ep^{2}}+2(\xi_{a}-\ol{x})
\Bigr), \\
D_{y}u^{a}(\ol{y},\ol{s})=& \, 
\frac{\ol{\mu}_{1a}}{\ol{\mu}_{a}}
\cdot\frac{\ol{y}-\ol{x}}{\ep^{2}}, \\
D_{z}v^{a}(\ol{z})=& \, 
\frac{\ol{\mu}_{1a}}{1-\ol{\mu}_{a}}
\cdot
\frac{\ol{z}-\ol{x}}{\ep^{2}}. 
\end{align*}
In view of the Lipschitz continuity of $u^{a}$ on $\ol{W}_{a}$, 
we see that  
$\bigl\{\frac{\ol{y}-\ol{x}}{\ep^{2}}\bigr\}_{\ep}$, 
$\bigl\{\frac{\ol{z}-\ol{x}}{\ep^{2}}\bigr\}_{\ep}$ 
are bounded uniformly in $\ep>0$. 
Thus we may assume that 
\begin{align*}
&\frac{\ol{y}-\ol{x}}{\ep^{2}}\to p_{y}^{a}, \quad 
\frac{\ol{z}-\ol{x}}{\ep^{2}}\to p_{z}^{a}, \\
&\ol{\mu}_{1a}\to\mu_{1a}:=u^{a}(\xi_{a},\sig_{a})-v^{a}(\xi_{a}), \\
&\ol{\mu}_{a}\to\mu_{a}(\sig_{a})
\end{align*}
as $\ep\to0$ for some $p_y^{a}, p_z^{a}\in\R^{N}$ 
by taking a subsequence if necessary. 
Set 
\begin{align*}
&P_{a}:=\frac{\mu_{1a}}{\mu_{a}(\sig_{a})}p_{y}^{a}, \\
&Q_{a}:=\frac{\mu_{1a}}{1-\mu_{a}(\sig_{a})}p_{z}^{a} 
\quad\textrm{and}\\
&\tilde{P}_{a}:=\mu_{a}(\sig_{a})(P_{a}-Q_{a}). 
\end{align*}

Sending $\ep\to0$ in inequalities \eqref{pf:mu-0}, 
\eqref{pf:mu-1} and \eqref{pf:mu-2} yields 
\begin{align}
-\eta+H(\xi_{a}, \tilde{P}_{a}
+Q_{a})\ge0, \nonumber\\
-\frac{1}{\mu_{a}(\sig_{a})}(\eta+\mu_{1a}\phi^{'}(\sig_{a}))
+H(\xi_{a}+ae_{N}, P_{a})\le0, \label{pf:mu-3}\\
H(\xi_{a}+ae_{N}, Q_{a})\le0. \nonumber
\end{align}
Note that $|P_{a}|+|Q_{a}|\le R$ 
for some $R>0$ which is independent of $a$. 
There exists a modulus $\om_{R}$ such that 
\begin{equation}\label{pf:mu-4}
H(\xi_{a}+ae_{N}, \tilde{P}_{a}+Q_{a})
\ge 
\eta-\om_{R}(a). 
\end{equation}
For small $a>0$ we have 
\[
H(\xi_{a}+ae_{N}, \tilde{P}_{a}+Q_{a})
\ge\frac{\eta}{2}. 
\]
Since we have 
$\tilde{P}_{a}/\mu_{a}(\sig_{a})+\tilde{Q}_{a}=P_{a}$, 
by using (A2)$_{0}^{+}$, we obtain 
\begin{equation}\label{pf:mu-5}
\mu_{a}(\sig_{a})
H(\xi_{a}+ae_{N}, P_{a})
\ge 
H(\xi_{a}+ae_{N}, \tilde{P}_{a}+Q_{a})
+\psi_{\eta}(1-\mu_{a}(\sig_{a})) 
\end{equation}
for a constant $\psi_{\eta}>0$. 

By \eqref{pf:mu-4}, \eqref{pf:mu-5} and \eqref{pf:mu-3} 
we get 
\begin{align*}
\eta-\om_{R}(a)\le&\, 
H(\xi_{a}+ae_{N}, \tilde{P}_{a}+Q_{a})\\
{}\le&\, 
\mu_{a}(\sig_{a})
H(\xi_{a}+ae_{N}, P_{a})-\psi_{\eta}(1-\mu_{a}(\sig_{a}))\\
{}\le&\, 
\eta+\mu_{1a}\phi^{'}(\sig_{a})-\psi_{\eta}(1-\mu_{a}(\sig_{a})). 
\end{align*}
We divide by $\mu_{1a}>0$ and then we obtain 
\begin{align*}
0&\le \, 
\phi^{'}(\sig_{a})+
\frac{\psi_{\eta}}{\mu_{1a}}(\mu_{a}(\sig_{a})-1)
+\frac{\om_{R}(a)}{\mu_{1a}}\\
{}&\le \, 
\phi^{'}(\sig_{a})+
\frac{\psi_{\eta}}{C}(\mu_{a}(\sig_{a})-1)
+\om_{R}(a). 
\end{align*}
Sending $a\to0$ yields 
\[
\phi^{'}(\sig)+
\frac{\psi_{\eta}}{C}(\mu(\sig)-1)\ge0, 
\]
which is the conclusion. 
\end{proof}

\subsection{Dirichlet Problem}
Let $u$ be the solution of (CD) 
and $c_{sc}$ be the constant defined by \eqref{additive-const-state}.

We first treat the case where $c_{sc}>0$. 
Set $u_{c_{sc}}(x,t)=u(x,t)+c_{sc}t$ 
for $(x,t)\in \cQ$ and 
$g_{c_{sc}}(x,t)=g(x)+c_{sc}t$ for $(x,t)\in\bQ$. 
Since $u_{c_{sc}}$ is bounded on $\cQ$ by 
Proposition \ref{prop:bound} (iii-a) and 
$g_{c_{sc}}(x,t)\to\infty$ uniformly for $x\in\bO$ as $t\to\infty$, 
there exists a constant $\ol{t}>0$ such that $g_{c_{sc}}(x,t)>u_{c_{sc}}(x,t)$ 
for all $(x,t)\in(\ol{t},\,\infty)$. 
As $u_{c_{sc}}$ satisfies
\[\left\{
\begin{array}{ll}
(u_{c_{sc}})_t+H(x,Du_{c_{sc}})=c_{sc} \ \ &\mbox{ in } \Q,\\
\noalign{\vspace{5pt}}
u_{c_{sc}}(x,t)=g_{c_{sc}}(x,t) &\mbox{ on } \bO\times(0,\infty) 
\end{array}
\right.\]
in the viscosity sense, 
we see easily that $u_{c_{sc}}$ is a solution of 
\[\left\{
\begin{array}{ll}
(u_{c_{sc}})_t+H(x,Du_{c_{sc}})\le c_{sc} \ \ &\mbox{ in }\Om\times (\ol{t},\infty),\\
\noalign{\vspace{5pt}}
(u_{c_{sc}})_t+H(x, Du_{c_{sc}})\ge c_{sc} &\mbox{ on }\cO\times(\ol{t},\infty).  
\end{array}
\right.\]
Therefore the large-time asymptotic behavior of $u_{c_{sc}}$ 
is same as the large-time asymptotic behavior of solutions of (SC).

Next we consider the case where $c_{sc}\le0$. 
Let $v$ be a subsolution of (D)$_{0}$. 
Since $u$ is bounded on $\cQ$ by Proposition \ref{prop:bound} (iii-b) and 
$v-M$ are still subsolutions of (D)$_{0}$ for 
any $M>0$,  
by subtracting a positive constant to $v$ if necessary, 
we may assume that 
$u-v$ satisfies \eqref{bdd:u-v}.

Let $\mu_{\eta}^{+}$ be the function defined by \eqref{def:mu-plus} 
for $\eta\in(0,\eta_{0}]$. 
By abuse of notation we write $\mu$
for $\mu_{\eta}^{+}$. 
We prove that $\mu$ satisfies \eqref{variational-ineq}. 
We first notice that, in view of the coercivity of $H$, 
we have 
\begin{align}
&
u(x,t)\le g(x) \ \textrm{and} \label{pf:cd-1}\\
&
v(x)\le g(x) \nonumber
\quad\textrm{for all} \ (x,t)\in\bQ 
\end{align}
by Proposition \ref{prop:classic}.

Let $\ep>0$, $\xi\in\cO$, 
$\sig, \tau\in(0,\infty)$ 
and $(\ol{x},\ol{t})\in\cQ$ 
be the same as those in the proof of Lemma \ref{lem:mu} 
in the case of (SC). 
We only consider the case where $\mu(\tau)<1$ 
and $\xi\in\bO$. 
Since we have 
\[
0\le 
\mu(\tau)=
\frac{u(\xi,\tau)-v(\xi)+\eta(\tau-\sig)}
{u(\xi,\sig)-v(\xi)}
<1
\]
and $u(\xi,\sig)-v(\xi)>0$, we obtain 
\[
u(\xi,\tau)-v(\xi)+\eta(\tau-\sig)
<
u(\xi,\sig)-v(\xi)\le g(\xi)-v(\xi), 
\]
which implies 
$u(\xi,\tau)<g(\xi)$.  
Therefore 
we have 
$u(\ol{x},\ol{t})-g(\ol{x})<0$ 
for small $\ep>0$. 
The rest of the proof follows by the same method as in 
subsection 4.2.

\section{Convergence}

In this section, we prove Theorem \ref{thm:large-time} by using 
Theorem \ref{thm:asymp-mono}.

\begin{proof}[Proof of Theorem {\rm \ref{thm:large-time}}]
Let $c$ be the associated additive eigenvalue. 
When we consider (CD), 
let us set 
$c:=c_{sc}$ if $c_{sc}>0$ and 
$c:=0$ if $c_{sc}\le0$.

Since $\{u_{c}(\cdot,t)\}_{t\ge0}$ is compact in $\W(\Om)$, 
there exists a sequence $\{u_{c}(\cdot,T_{n})\}_{n\in\N}$ which 
converges uniformly on $\cO$ as $n\to\infty$. 
The maximum principle implies that we have 
\[
\|u_{c}(\cdot,T_{n}+\cdot)-u_{c}(\cdot,T_{m}+\cdot)\|_{L^{\infty}(\Q)}
\le 
\|u_{c}(\cdot,T_{n})-u_{c}(\cdot,T_{m})\|_{L^{\infty}(\Om)} 
\]
for any $n,m\in\N$. 
Therefore, $\{u_{c}(\cdot,T_{n}+\cdot)\}_{n\in\N}$ 
is a Cauchy sequence in $\BUC(\cQ)$ and it converges to a function denoted by 
$u_{c}^{\infty}\in\BUC(\cQ)$.

Fix any $x\in\cO$ and $s,t\in[0,\infty)$ with $t\ge s$.  
By Theorem \ref{thm:asymp-mono} we have 
\[
u_{c}(x,s+T_{n})-u_{c}(x,t+T_{n})+\eta(s-t)
\le \del_{\eta}(s+T_{n}) 
\]
or 
\[
u_{c}(x,t+T_{n})-u_{c}(x,s+T_{n})-\eta(t-s)
\le \del_{\eta}(s+T_{n}) 
\]
for any $n\in\N$ and $\eta>0$. 
Sending $n\to\infty$ and then $\eta\to0$, we get, for any $t\ge s$
\[
u_{c}^{\infty}(x,s)\le u_{c}^{\infty}(x,t). 
\]
or 
\[
u_{c}^{\infty}(x,t)\le u_{c}^{\infty}(x,s). 
\]
Therefore, we see that the functions $x\mapsto u_{c}^{\infty}(x,t)$ are uniformly bounded and equi-continuous which are also monotone in $t$. This implies that $u_{c}^{\infty}(x,t)\to w(x)$ uniformly on $\cO$ as $t\to\infty$
for some $w\in \W(\Om)$. Moreover, by a standard stability property of viscosity solutions, 
$w$ is a solution of either (E) or (D)$_{0}$.

Since $u_{c}(\cdot,T_{n}+\cdot)\to u_{c}^{\infty}$ uniformly in $\cQ$ 
as $n\to\infty$, we have 
\[
-o_n(1)+u_{c}^{\infty}(x,t)
\le 
u_{c}(x,T_{n}+t)
\le 
u_{c}^{\infty}(x,t)+o_n(1), 
\]
where $o_n(1)\to\infty$ as $n\to\infty$, uniformly in $x$ and $t$.
Taking the half-relaxed semi-limits as $t \to +\infty$, 
we get 
\[
-o_n(1)+v(x)
\le 
\limiinf_{t\to\infty}[u_{c}](x,t)
\le 
\limssup_{t\to\infty}[u_{c}](x,t)
\le 
v(x)+o_n(1). 
\]
Sending $n\to\infty$ yields 
\[
v(x)
=
\limiinf_{t\to\infty}[u_{c}](x,t)
=
\limssup_{t\to\infty}[u_{c}](x,t)
\]
for all $x\in\cO$. 
\end{proof}

\begin{rem}\label{rem:general}\ \\
(i) The Lipschitz regularity assumption 
on $u_{0}$ is convenient to avoid technicalities but it is not necessary. 
We can remove it as follows. 
We may choose a sequence 
$\{u_{0}^{k}\}_{k\in\N}\subset \W(\Om)\cap C(\cO)$ 
so that $\|u_{0}^{k}-u_{0}\|_{\Li(\Om)} \le 1/k$ 
for all $n\in\N$. 
By the maximum principle, we have 
\[
\|u-u_{k}\|_{\Li(\Q)}
\le 
\|u_{0}^{k}-u_{0}\|_{\Li(\Om)} \le 1/k 
\]
and therefore 
\[
u_{k}(x,t)-1/k\le u(x,t)\le u_{k}(x,t)+1/k 
\ \textrm{for all} \ (x,t)\in\cQ, 
\]
where $u$ is the solution of (IB) 
and $u_{k}$ are the solutions 
of (IB) with $u_{0}=u_{0}^{k}$. 
Therefore we have 
\[
u_{\infty}^{k}(x)-1/k
\le 
\limiinf_{t\to\infty}u(x,t)
\le 
\limssup_{t\to\infty}u(x,t)
\le 
u_{\infty}^{k}(x)+1/k
\]
for all $x\in\cO$, 
where $u_{\infty}^{k}(x)=\lim_{t\to\infty}u_{k}(x,t)$. 
Thus, 
$
|\limiinf_{t\to\infty}u(x,t)
-
\limssup_{t\to\infty}u(x,t)|\le2/k
$
for all $k\in\N$ and $x\in\cO$, which implies that 
\[
\limiinf_{t\to\infty}u(x,t)
=
\limssup_{t\to\infty}u(x,t)
\]
for all $x\in\cO$. \smallskip
We note that, by the same argument, we can obtain the asymptotic monotone 
property, Theorem \ref{thm:asymp-mono}, of solutions of (IB) 
without the Lipschitz continuity of solutions.

\noindent
(ii) 
We remark that modifying (A2)$_{a}^{\pm}$ as in \cite{BS}, 
we can generalize Theorem \ref{thm:large-time}. 
We use the following assumptions instead of (A2)$_{a}^{\pm}$. 
We denote $H-a$ by $H_{a}$. 
\begin{itemize}
\item[{\rm(A3)}$_{a}^{+}$] 
There exists a closed set $K\subset\cO$ 
($K$ is possibly empty) having the properties\\
\begin{itemize}
\item[{\rm (i)}] 
$\min_{p\in\R^{N}}H_{a}(x,p)=0$ for all $x\in K$, \\
\item[{\rm (ii)}] 
for each $\ep>0$ there exists a modulus $\psi_{\ep}(r)>0$ 
for all $r>0$ and $\eta_{0}^{\ep}>0$ such that 
for all $\eta\in(0,\eta_{0}^{\ep}]$ 
if $\dist(x,K)\ge\ep$, $H_{a}(x,p+q)\ge\eta$ and 
$H_{a}(x,q)\le0$ for some 
$x\in\cO$ and $p,q\in\R$, then for any $\mu\in(0,1]$, 
\[
\mu H_{a}(x,\frac{p}{\mu}+q)\ge 
H_{a}(x,p+q)+\psi_{\ep}(\eta)(1-\mu). 
\]
\end{itemize}
\item[{\rm(A3)}$_{a}^{-}$] 
There exists a closed set $K\subset\cO$ 
($K$ is possibly empty) having the properties\\
\begin{itemize}
\item[{\rm (i)}] 
$\min_{p\in\R^{N}}H_{a}(x,p)=0$ for all $x\in K$, \\
\item[{\rm (ii)}] 
for each $\ep>0$ there exists a modulus $\psi_{\ep}(r)>0$ 
for all $r>0$ and $\eta_{0}^{\ep}>0$ such that 
for all $\eta\in(0,\eta_{0}^{\ep}]$ 
if $\dist(x,K)\ge\ep$, $H_{a}(x,p+q)\le-\eta$ and 
$H_{a}(x,q)\ge0$ for some 
$x\in\cO$ and $p,q\in\R$, then for any $\mu\in(0,1]$, 
\[
\mu H_{a}(x,\frac{p}{\mu}+q)\le 
H_{a}(x,p+q)-\frac{\psi_{\ep}(\eta)(\mu-1)}{\mu}. 
\]
\end{itemize}
\end{itemize}

\begin{thm}\label{thm:large-time-general}
The results of Theorem {\rm \ref{thm:large-time}} still
hold if we replace assumptions  {\rm (A2)$_{c}^{+}$} or {\rm (A2)$_{c}^{+}$} by {\rm (A3)$_{c}^{+}$} or {\rm (A3)$_{c}^{+}$} where $c$ is defined as in Theorem {\rm \ref{thm:large-time}}. 
\end{thm}
\end{rem}

\section{Remarks on Convex Hamilton-Jacobi Equations}
In this section we deal with convex Hamilton-Jacobi equations, 
i.e., 
\begin{itemize}
\item[{\rm (A4)}] 
$p\mapsto H(x,p)$ is convex for any $x\in\cO$. 
\end{itemize}

\subsection{The Namah-Roquejoffre Case}

We use the following assumptions in this subsection. 
We consider the Hamiltonian $H(x,p) = F(x,p) - f(x)$, 
where $F$ and $f$ are assumed to satisfy 
 \begin{itemize}
\item[{\rm (A5)}] 
$F(x,p)\ge F(x,0)=0$
for all $(x,p)\in\cO\times\R^{N}$, 
\item[{\rm (A6)}] 
$f(x)\ge0$ for all $x\in\cO$ and 
$\cA_{f}:=\{x\in\cO\mid f(x)=0\}\not=\emptyset$. 
\item[{\rm (A7)}] 
$g(x)\ge0$ for all $x\in\bO$. 
\end{itemize}
A typical example of this Hamiltonian is 
$H(x,p)=|p|-f(x)$ with $f\ge0$ on $\cO$ 
and in this case, 
it is clearly seen that 
$H$ does not satisfy (A2)$_{0}^{+}$ and (A2)$_{0}^{-}$. 
Meanwhile, Hamiltonian has a simple structure 
and therefore in a relatively easy way we can obtain 
\begin{thm}\label{thm:NR-conv}
Assume that {\rm (BA)} and {\rm (A4)-(A6)} hold.\\ 
{\rm (i)} {\bf (Neumann/oblique derivative problem)} 
Assume that {\rm (A7)} holds. 
For the solution $u$ of {\rm (CN)}, 
\eqref{conv} holds with a solution $(v,0)$ of {\rm (E-N)}. \\
{\rm (ii)} {\bf (State constraint problem or Dirichlet problem)} 
For the solution $u$ of {\rm (SC)} or {\rm (CD)}, 
\eqref{conv} holds with a solution 
$(v,0)$ of {\rm (E-S)} or {\rm (D)}$_{0}$.
\end{thm}

\begin{proof}[Sketch of Proof] 
It is easy to see that 
\[
c_{n}=0 \ \textrm{and} \ 
c_{sc}=0. 
\]
Indeed, on one hand, in view of (A5)-(A7), 
any constant is a subsolution of (E-N) and (E-SC) 
with $a=0$, which implies that 
$c_{n}, c_{sc}\le0$. 
On the other hand,
if we assume $c_{n}<0$ or $c_{sc}<0$, 
then (A5), (A6) yields a contradiction since 
$H(x,p)=F(x,p)-f(x)\ge0$ for all $x\in\cA_{f}$.

We have $u_{t}=-F(x,Du)+f(x)\le0$ in 
$\cA_{f}\times(0,\infty)$ 
and therefore 
we can get the monotonicity of 
the function $t\mapsto u(x,t)$
at least formally 
(see \cite[Lemma 2.4]{NR} and \cite[Lemma 4.4]{GLM1} 
for a rigorous proof). 
Thus, taking into account the uniform Lipschitz continuity of $u$, we have 
$\limiinf_{t\to\infty}u(x,t)
=
\limssup_{t\to\infty}u(x,t)$ 
for any $x\in\cA_{f}$. 
By \cite[Theorem 6.6]{I2} for Neumann problems, 
\cite[Theorem 7.3]{IM} for state constraint problems and 
\cite[Theorem 5.3]{M2}, we obtain 
$\limiinf_{t\to\infty}u(x,t)
=
\limssup_{t\to\infty}u(x,t)$ 
for any $x\in\cO$, since, outside $\cA_{f}$, one has a strict subsolution, 
which is a key tool to obtain the comparison of the half-relaxed limits. 
\end{proof}

\begin{rem}
We remark that the result of Theorem \ref{thm:NR-conv} 
is included in Theorem \ref{thm:large-time-general}, 
since {\rm (A3)}$_{0}^{+}$ holds with 
$K=\{x\in\R^{N}\mid f(x)=0\}$. 
\end{rem}

\subsection{Asymptotic Profile}

In this subsection we give representation formulas for the
asymptotic solutions in each case. 
We define the functions $\phi^{-}, \phi^{\infty}\in C(\cO)$ 
by 
\begin{align*}
\phi^{-}(x)&:=\sup\{v(x)\mid v\in C(\cO) \ 
\textrm{is a subsolution of} \ {\rm (E)}, \ 
v\le u_{0} \ \textrm{on} \ \cO\}, \\
\phi^{\infty}(x)&:=\inf\{v(x)\mid v\in C(\cO) \ 
\textrm{is a solution of} \ {\rm (E)}, \ 
v\ge \phi^{-} \ \textrm{on} \ \cO\}.  
\end{align*}
We denote $\phi^{-}$ associated 
with (E-N) and (E-SC) by $\phi^{-}_{n}$ and $\phi^{-}_{s}$, 
respectively and 
$\phi^{\infty}$ associated 
with (E-N) and (E-SC) by $\phi^{\infty}_{n}$ and $\phi^{\infty}_{s}$, 
respectively. 
In view of (A4) 
we see that 
$\phi^{\infty}_{n}$ and 
$\phi^{\infty}_{s}$ 
are a solution of (E-N), (E-SC), respectively. 
We refer to the articles \cite{BJ, IM, I2} for 
a stability result under infimum operation.  
When $c_{sc}\le0$, 
we define the functions $\phi_{d}^{-}, \phi^{\infty}_{d}
\in C(\cO)$ by 
\begin{align*}
\phi_{d}^{-}(x)&:=\sup\{v(x)\in C(\cO)\mid 
v\in C(\cO) \ 
\textrm{is a subsolution of} \ {\rm (D)_{0}}\}, \\
\phi^{\infty}_{d}(x)
&:=\inf\{v(x)\mid v\in C(\cO) \ 
\textrm{is a solution of} \ {\rm (D)_{0}}, \ 
v\ge \phi^{-}_{s}\land \phi_{d}^{-}\ \textrm{on} \ \cO\}. 
\end{align*}

\begin{thm}[Asymptotic Profile]\label{thm:profile}
Assume that {\rm (BA)} and {\rm (A4)} hold.\\ 
{\rm (i)} 
{\bf (Neumann/oblique derivative problem)} 
Let $u$ be the solution of {\rm (CN)} and then we have 
\begin{equation}\label{profile-1}
\lim_{t\to\infty}(u(x,t)+c_{n}t) 
=\phi^{\infty}_{n}(x)
\quad\textrm{uniformly for all} \ 
x\in\cO.  
\end{equation}
{\rm (ii)} 
{\bf (State constraint problem)} 
Let $u$ be the solution of {\rm (SC)} and then we have 
\begin{equation}\label{profile-2}
\lim_{t\to\infty}(u(x,t)+c_{sc}t) 
=\phi^{\infty}_{s}(x)
\quad\textrm{uniformly for all} \ 
x\in\cO.  
\end{equation}
{\rm (iii)} 
{\bf (Dirichlet problem)} 
Let $u$ be the solution of {\rm (CD)}. \\
{\rm (iii-a)} 
If $c_{sc}>0$, then 
\begin{equation}\label{profile-3}
\lim_{t\to\infty}(u(x,t)+c_{sc}t)=\phi^{\infty}_{s}(x) 
\quad\textrm{uniformly for all} \ 
x\in\cO. 
\end{equation}
{\rm (iii-b)} 
If $c_{sc}=0$, then 
\begin{equation}\label{profile-4}
\lim_{t\to\infty}u(x,t)=\phi^{\infty}_{d}(x) 
\quad\textrm{uniformly for all} \ 
x\in\cO. 
\end{equation}
{\rm (iii-c)} 
If $c_{sc}<0$, then 
\[
\lim_{t\to\infty}u(x,t)=\phi_{d}^{-}(x) 
\quad\textrm{uniformly for all} \ 
x\in\cO. 
\]
\end{thm}
\begin{proof}
We prove (i), (ii), (iii-a) and (iii-b) 
at the same time in order to avoid duplication 
of explanations. 
We denote $\phi^{-}_{n}$, $\phi^{-}_{s}$ 
and $\phi^{-}_{s}\land\phi_{d}^{-}$ 
by $\phi^{-}$, and 
$\phi^{\infty}_{s}$, $\phi^{\infty}_{n}$ and $\phi^{\infty}_{d}$ by 
$\phi^{\infty}$ 
in any case. 
We write $u_{\infty}(x)$ for the left hand side of 
\eqref{profile-1}, \eqref{profile-2}, 
\eqref{profile-3} and \eqref{profile-4} 
in any case. 
Let $c$ denote $c_{n}$ and $c_{sc}$ in any case 
and set $u_{c}:=u+ct$ on $\cQ$.

Since $\phi^{-}$ is a subsolution of (E-N), (E-SC) 
or (D)$_{c}$, respectively, 
$\phi^{-}$ is a subsolution of (IB) too. 
By the comparison principle for 
(IB) (cf. Theorem \ref{thm:comparison}) 
we have 
\begin{equation}\label{phi-ineq}
\phi^{-}(x)\le u_{c}(x,t)
\quad\textrm{for all} \ 
(x,s)\in\cQ.  
\end{equation}
Therefore, we get 
\[\phi^{-}(x)\le u_{\infty}(x)
\quad\textrm{for all} \ 
x\in\cO. 
\]
Note that $u_{\infty}$ is a solution of (E-N), (E-SC) or 
(D)$_{c}$. 
By the definition of $\phi^{\infty}$ 
we get $\phi^{\infty}\le u_{\infty}$ on $\cO$.

We define the functions $v_{c}\in\BUC(\cQ)$ 
by $v_{c}(x,t):=\inf_{s\ge t}u_{c}(x,s)$. 
By \eqref{phi-ineq} we have particularly 
$\phi^{-}\le v_{c}(\cdot,0)$ on $\cO$. 
Note that $v_{c}(\cdot,0)$ is a subsolution of (E-N), (E-SC) or 
(D)$_{c}$ and 
$v_{c}(\cdot,0)\le u_{0}$ on $\cO$. 
Indeed, 
$v_{c}$ satisfies $(v_{c})_t+H(x,Dv_{c}(x,t))=0$ in $\Q$ 
(see \cite{BJ, IM, I2}) and 
it is clear that $v_{c}(x,\cdot)$ is non-decreasing, 
from which we have 
$(v_{c})_t(x,t)\ge0$ in $\Q$ in the viscosity sense. 
Therefore $H(x,Dv_{c}(x,t))\le 0$ in $\Om$ for any $t\ge0$. 
Since $v_{c}(x,t)\to v_{c}(x,0)$ uniformly on $\cO$ as 
$t\to 0$, 
we see that $H(x,Dv_{c}(x,0))\le 0$ in $\Om$ in the viscosity sense. 
It is easily seen that 
$v_{c}(\cdot,0)\le u_{0}$ on $\cO$. 
Thus by the definition of $\phi^{-}$, 
we have $v_{c}(\cdot,0)\le \phi^{-}$ on $\cO$. 
Therefore we obtain $v_{c}(\cdot,0)=\phi^{-}$ on $\cO$.

Note that $\phi^{\infty}$ is a solution of \eqref{ib-1}, \eqref{ib-2} and 
satisfies that $\phi^{\infty}(x)\ge\phi^{-}(x)=v_{c}(x,0)$ on $\cO$. 
By the comparison principle for (IB)
(cf. Theorem \ref{thm:comparison}, again) 
we have for all $x\in\cO$,  
\[u_\infty(x)=\lim_{t\to\infty}u_{c}(x,t)
=\liminf_{t\to\infty}u_{c}(x,t)
=\lim_{t\to\infty}v_{c}(x,t)\le\phi^{\infty}(x),\]
which establishes formulas. 

We finally consider (iii-c). 
In this case  
Problem (D)$_0$ has the unique solution. 
Since $u_{\infty}$ and $\phi_{d}$ are  
solutions of (D)$_0$, we see that $u_{\infty}=\phi_{d}$ 
on $\cO$. 
\end{proof}

We finally give another formulas for 
$\phi^{-}_{n}, \phi^{-}_{s}, \phi_{d}^{-}$, 
$\phi^{\infty}_{n}, \phi^{\infty}_{s}$ 
and $\phi^{\infty}_{d}$. 
\begin{prop}
Assume that {\rm (BA)} and {\rm (A4)} hold.
We have 
\begin{align}
&\phi^{-}_{n}(x)=
\min\{d_{n}(x,y)+u_{0}(y)\mid y\in\cO\}, \label{ano-1}\\
&\phi^{\infty}_{n}(x)=
\min\{d_{n}(x,y)+\phi^{-}_{n}(y)\mid y\in\cA_{n}\}, \label{ano-2}\\
&\phi^{-}_{s}(x)=
\min\{d_{s}(x,y)+u_{0}(y)\mid y\in\cO\}, \label{ano-3}\\
&\phi^{\infty}_{s}(x)=
\min\{d_{s}(x,y)+\phi^{-}_{s}(y)\mid y\in\cA_{s}\}, \label{ano-4}\\
&\phi_{d}^{-}(x)=
\min\{d_{0}(x,y)+g(y)\mid y\in\bO\}, \label{ano-5}\\
&\phi_{d}^{\infty}(x)=
\min\{d_{0}(x,y)+\phi_{s}^{-}(y)\land\phi_{d}^{-}(y)\mid y\in\cA_{s}\} 
\label{ano-6}
\end{align}
for any $x\in\cO$, 
where 
\begin{align*}
&d_{n}(x,y):=
\sup\{v(x)-v(y)\mid v \ \textrm{is a subsolution of} 
\ \eqref{erg-1}, \eqref{erg-2} \ \textrm{with} \ a=c_{n}\},  \\
&d_{s}(x,y):=
\sup\{v(x)-v(y)\mid v \ \textrm{is a subsolution of} \ 
\eqref{erg-1} \ \textrm{with} \ a=c_{sc}\}, \\
&d_{0}(x,y):=
\sup\{v(x)-v(y)\mid v \ \textrm{is a subsolution of} \ \eqref{d-1} 
\ \textrm{with} \ a=0\}, \\
&\cA_{n}:=
\{y\in\cO\mid d_{n}(\cdot,y) \ \textrm{is a solution of} 
\ {\rm (E}\textrm{-}{\rm N)} \}, \\
&\cA_{s}:=
\{y\in\cO\mid d_{s}(\cdot,y) \ \textrm{is a solution of} 
\ {\rm (E}\textrm{-}{\rm SC)} \}.  
\end{align*}
\end{prop}
\begin{proof}
We first see that \eqref{ano-5} holds. 
We denote by $w_{d}$ the right hand side 
of \eqref{ano-5}. 
By a standard stability result of viscosity solution 
$\phi_{d}^{-}$ is a subsolution of (D)$_{0}$. 
By Proposition \ref{prop:classic} we have 
$\phi_{d}^{-}\le g$ on $\bO$. 
Therefore we have 
$\phi_{d}^{-}(x)\le d_{0}(x,y)+\phi_{d}^{-}(y)\le d_{0}(x,y)+g(y)$ 
for any $x,y\in\cO$, 
which implies that $\phi_{d}^{-}\le w_{d}$ on $\cO$. 
It is easily seen that $w_{d}$ is a subsolution of 
(D)$_{0}$ and therefore $w_{d}\le\phi_{d}^{-}$ on $\cO$.

We next prove \eqref{ano-1} and \eqref{ano-3}. 
We denote $\phi^{-}_{n}$, $\phi^{-}_{s}$ 
and $\phi^{-}_{s}\land\phi_{d}^{-}$ 
by $\phi^{-}$ and 
denote by $w_{-}$ the right hand side 
of \eqref{ano-1} and \eqref{ano-3} in any case. 
Let $d$ denote $d_{n}$ or $d_{s}$ in any case. 
By the definition of $\phi^{-}$ we see that 
$\phi^{-}$ is a subsolution of (E-N), (E-SC) or (D)$_{0}$ and 
$\phi^{-}\le u_{0}$ on $\cO$. 
By the definition of $d$ we have 
$
\phi^{-}(x)
\le d(x,y)+\phi^{-}(y)
\le d(x,y)+u_{0}(y)
$
for all $x,y\in\cO$, which implies 
$\phi^{-}\le w_{-}$ on $\cO$. 
Note that 
$w_{-}$ is a subsolution of (E-N), (E-SC) or (D)$_{0}$ and 
$w_{-}(x)\le d(x,x)+u_{0}(x)=u_{0}(x)$ 
for all $x\in\cO$. 
By the definition of $\phi^{-}$
we obtain $\phi^{-}=w_{-}$ on $\cO$.

We finally prove \eqref{ano-2}, \eqref{ano-4} 
and \eqref{ano-6}. 
We denote $\phi^{\infty}_{n}$, $\phi^{\infty}_{s}$ 
and $\phi^{\infty}_{d}$ by $\phi^{\infty}$ and 
denote by $w_{\infty}$ the right hand side 
of \eqref{ano-2}, \eqref{ano-4} and \eqref{ano-6} 
in any case. 
Let $\cA$ denote $\cA_{n}$ or $\cA_{s}$ in any case.

It is easy to see that 
$w_{\infty}(x)=\phi^{-}(x)$ for all $x\in\cA$. 
By \cite[Theorem 6.6]{I2}, \cite[Theorem 7.3]{IM} and 
\cite[Theorem 5.3]{M2} 
we get $w_{\infty}\ge \phi^{-}$ on $\cO$. 
By the definition of $\phi^{\infty}$ we obtain 
$w_{\infty}\ge \phi^{\infty}$ on $\cO$.
Note that $\phi^{-}\le\phi^{\infty}$ on $\cO$. 
Then, we have 
\[
w_{\infty}(x)=\phi^{-}(x)\le\phi^{\infty}(x) 
\quad\textrm{for all} \ x\in\cA. 
\]
Therefore we get $w_{\infty}\le\phi^{\infty}$ on $\cO$ 
by \cite[Theorem 6.6]{I2}, \cite[Theorem 7.3]{IM} and 
\cite[Theorem 5.3]{M2}. 
\end{proof}

\subsection{On the compatibility condition \eqref{comp-cond}}
In this subsection we consider (CD) 
under assumptions (A1), (A4) and {\rm(A2)}$_{c}^{+}$ or {\rm(A2)}$_{c}^{-}$ 
and we remove the compatibility condition \eqref{comp-cond}, 
where 
$c:=c_{sc}$ if $c_{sc}>0$ and $c:=0$ if $c_{sc}\le0$. 
Let $u$ be a solution of (CD). 
We notice that, 
if we do not assume \eqref{comp-cond}, 
then $u$ may be discontinuous 
and therefore we interpret solutions as discontinuous viscosity 
solutions introduced in \cite{I0}.

We consider the following approximate problems of (CD) 
\begin{numcases}
{\textrm{(CD)}_{k}^{1} \hspace{1cm}}
u_t +H(x,Du)= 0 & in $\Q$, \nonumber \\
u(x,t)=g(x) & on $\bQ$, \nonumber\\
u(x,0)=u_{0}^{k}(x) & on $\cO$, \nonumber
\end{numcases}
and 
\begin{numcases}
{\textrm{(CD)}_{k}^{2} \hspace{1cm}}
u_t +H(x,Du)= 0 & in $\Q$, \nonumber \\
u(x,t)=g_{k}(x,t) & on $\bQ$, \nonumber\\
u(x,0)=u_{0}(x) & on $\cO$, \nonumber
\end{numcases}
where 
\begin{align*}
u_{0}^{k}(x)&:=\, 
\min_{y\in\cO}\{u_{0}^{g}(y)+k|x-y|^{2}\}, \\
u_{0}^{g}(x)&:=\, 
\left\{
\begin{array}{lcl}
u_{0}(x) & \textrm{for} 
\ x\in \Om, \\
u_{0}(x)\land g(x) & \textrm{for} \ x\in\bO,  
\end{array}
\right. \\
g_{k}(x,t)&:=\, 
\max\{g(x), u_{0}(x)-kt\} 
\end{align*}
for all $x\in\cO$, $g\ge0$ and $k\in\N$. 
Note that 
\[
u_{0}^{k}(x)\le g(x) \ \textrm{and} \ 
u_{0}(x)\le g_{k}(x,t)
\]
for all $(x,t)\in\bO\times[0,\infty)$ and 
$k\in\N$.

Let $u_{k}^{1}$ and $u_{k}^{2}$ be the solution 
of (CD)$_{k}^{1}$ and (CD)$_{k}^{2}$, respectively. 
By Theorem \ref{thm:profile} 
we have for any $k\in\N$ 
\[
\left\{
\begin{array}{ll}
\textrm{if} \ c_{sc}>0, 
& u_{k}^{1}(\cdot,t)+c_{sc}t\to 
\min\{d_{s}(\cdot,y)+\phi^{-}_{k}(y)\mid y\in\cA_{s}\}\\ 
\textrm{if} \ c_{sc}=0, 
& u_{k}^{1}(\cdot,t)\to 
\min\{d_{s}(\cdot,y)+\phi^{-}_{k}(y)\land\phi_{d}^{-}(y)
\mid y\in\cA_{s}\}, \\ 
\textrm{if} \ c_{sc}<0, 
& u_{k}^{1}(\cdot,t)\to\phi_{d}^{-} 
\end{array}
\right.
\]
uniformly on $\cO$ as $t\to\infty$, where 
\[
\phi^{-}_{k}(x):=\min\{d_{s}(x,y)+u_{0}^{k}(y)\mid y\in\cO\}
\ \textrm{for all} \ x\in\cO. 
\] 
By \cite[Theorem 6.1]{M3} we have for any $k\in\N$ 
\[
\left\{
\begin{array}{ll}
\textrm{if} \ c_{sc}\ge0, 
& u_{k}^{2}(\cdot,t)+c_{sc}t\to 
\min\{d_{s}(\cdot,y)+\phi^{-}_{s}(y)\land\phi^{-}_{g_k}(y)\mid 
y\in\cA_{s}\}\\ 
\textrm{if} \ c_{sc}<0, 
& u_{k}^{2}(\cdot,t)\to\phi_{d}^{-} 
\end{array}
\right.
\]
uniformly on $\cO$ as $t\to\infty$, where 
\begin{align*}
\phi^{-}_{g_k}(x)&:=\, 
\inf\{d_{s}(x,y)+\ul{g}_{k}(y)\mid y\in\bO\} 
\ \textrm{for all} \ x\in\cO,  
\\
\ul{g}_{k}(x)&:=\, 
\inf\{g_{k}(x,s)+c_{sc}s\mid s\ge0\}
\ \textrm{for all} \ x\in\bO.  
\end{align*}

\begin{prop}\label{prop:comp-cond1}
We have 
$\phi^{-}_{k}\to\phi_{s}^{-}$ uniformly on $\cO$ as $k\to\infty$. 
\end{prop}
\begin{proof}
It is easy to see that $\phi^{-}_{k}(x)$ is nondecreasing 
as $k\to\infty$ for any $x\in\cO$. 
We prove that 
$\phi^{-}_{k}(x)$ converges to 
$\phi^{-}_{s}(x)\land\phi_{d}^{-}(x)$ as $k\to\infty$ 
for any $x\in\cO$. 
Fix $x\in\cO$. 
First, we can easily see that 
$\limsup_{k\to\infty}\phi^{-}_{k}(x)\le \phi^{-}_{s}(x)\land \phi_{d}^{-}(x)$, 
since $\phi_{k}^{-}\le u_{0}^{k}\le u_{0}^{g}$ on $\cO$.

Choose $y_{k}, z_{k}\in\cO$ such that 
\begin{align*}
\phi^{-}_{k}(x)&=\, 
d_{s}(x,y_{k})+u_{0}^{k}(y_{k}) \ \textrm{and} \\
u_{0}^{k}(y_{k})&=\, 
u_{0}^{g}(z_{k})+k|y_{k}-z_{k}|^{2}. 
\end{align*}
Since $\cO$ is compact, we may assume that 
$y_{k}, z_{k}\to y_{0}\in\cO$ 
as $k\to\infty$ 
if taking a subsequence if necessary. 
We have 
\begin{align*}
\liminf_{k\to\infty}\phi^{-}_{k}(x)
&=\, 
\liminf_{k\to\infty}\bigl(d_{s}(x,y_{k})+u_{0}^{k}(y_{k})\bigr)\\
{}&\ge\, 
\liminf_{k\to\infty}
\bigl(d_{s}(x,y_{k})+u_{0}^{g}(z_{k})\bigr)\\
{}&\ge\, 
d_{s}(x,y_{0})+u_{0}^{g}(y_{0})\\
{}&\ge\, 
\phi^{-}_{s}(x)\land\phi_{d}^{-}(x). 
\end{align*}
Therefore we obtain 
$\limsup_{k\to\infty}\phi^{-}_{k}(x)
=\liminf_{k\to\infty}\phi^{-}_{k}(x)
=\phi^{-}_{s}(x)\land\phi_{d}^{-}(x)$. 
Since $\phi^{-}_{s}\land\phi_{d}^{-}\in C(\cO)$, 
in view of Dini's theorem, we see that 
the convergence is uniform for all $x\in\cO$. 
\end{proof}

\begin{prop}\label{prop:comp-cond2}
We have 
$\phi^{-}_{g_{k}}\to\phi_{d}^{-}$ uniformly on $\cO$ as $k\to\infty$. 
\end{prop}
\begin{proof}
We only need to consider $c_{sc}\ge0$ and 
prove that 
$\ul{g}_{k}\to g$ 
uniformly on $\bO$ as $k\to\infty$. 
Fix $x\in\bO$. 
Since 
$g_{k}(x,s)\ge g(x)$ for all $x\in\bO$ and $s\ge0$, 
we have 
$\ul{g}_{k}(x)
=\inf_{s\ge0}\{g_{k}(x,s)+c_{sc}s\}\ge g(x)$. 
Thus we have 
$\liminf_{k\to\infty}\ul{g}_{k}(x)\ge g(x)$ 
for all $x\in\bO$.

Fix any $s\in(0,\infty)$. 
Then there exists $k_{0}\in\N$ such that 
for all $k\ge k_{0}$, $g_{k}(x,s)=g(x)$. 
We get 
$\ul{g}_{k}(x)\le g_{k}(x,s)+c_{sc}s
=g_{k}(x)+c_{sc}s$. 
Therefore we have 
$\limsup_{k\to\infty}\ul{g}_{k}(x)
\le g(x)+c_{sc}s$. 
Sending $s\to0$, we obtain 
$\limsup_{k\to\infty}\ul{g}_{k}(x)\le g(x)$. 
Noting that $g\in C(\bO)$ and $\ul{g}_{k}$ is nonincreasing 
as $k\to\infty$, in view of Dini's theorem 
we get a conclusion. 
\end{proof}

By the comparison principle for (IB) we have 
\[
u_{k}^{1}\le u\le u_{k}^{2} \ \textrm{on} \ \cQ 
\ \textrm{for all} \ k\in\N. 
\]
By Propositions \ref{prop:comp-cond1}, \ref{prop:comp-cond2} 
we obtain 
\[
\limssup_{t\to\infty}u 
=
\limiinf_{t\to\infty}u
=
\lim_{k\to\infty}u^{1}_{k}
=
\lim_{k\to\infty}u^{2}_{k}
 \ \textrm{on} \ \cO. 
\]

\section{Appendix : Existence, Uniqueness and Regularity Results for {\rm (IB)}}

All the results presented in this appendix may appear at first glance as being well-known and covered by
the standard results of the theory of viscosity solutions 
(see for instance \cite{L2, I2, B2, CL}). But some of them are not completely
standard because we are dealing with an oblique vector field $\gam$ which is only continuous (and not Lipschitz continuous as in the standard cases) and with a domain $\Om$ which is not very regular. 
Of course, we can extend the classical results to this more general framework because the solutions of {\rm (IB)} are expected to be in $W^{1,\infty}$ since the Hamiltonian is coercive. But we have also to prove directly the existence of such $W^{1,\infty}$ -solutions, by using only comparison results which hold for $W^{1,\infty}$ sub and supersolutions, which is a little bit unusual in the theory of viscosity solutions.

Also, we remark that the calculations in the proof of comparison principles 
are used in the proofs of a key ingredient, Lemma \ref{lem:mu}, 
in order to prove the asymptotic monotone property, 
Theorem \ref{thm:large-time}. 
Therefore, reminding the proofs of the comparison principles 
helps us to understand them.

Before providing these results and their proofs, we recall, for the reader's convenience, the definition of 
viscosity solutions for (IB), and in particular of boundary conditions in the viscosity sense
(see \cite{CIL} for instance).
\begin{defn}\label{def} 
An upper-semicontinuous function $u$ 
{\rm (}resp., a lower semicontinuous function $u${\rm )} 
is a subsolution {\rm (}resp., supersolution{\rm )} of {\rm (IB)}
if the following conditions hold:\\
{\rm (i)} $u$ is a {\rm(}viscosity{\rm)} subsolution 
{\rm(}resp., {\rm(}viscosity{\rm)} supersolution{\rm)} of \eqref{ib-1}, \\
{\rm (ii)} 
$u(x, 0)\le u_{0}(x)$ 
{\rm (}resp., $u(x,0)\ge u_{0}(x)${\rm )} 
for all $x\in\cO$, 
and
\\{\rm (iii)} for any $\phi\in C^1(\cQ)$ 
and any $(x_0,t_0)\in\bQ$ such that 
$u-\phi$ takes a local maximum {\rm (}resp., minimum{\rm )}
at $(x_0,t_0)$, 
$$
\min\{\phi_t(x_0,t_0)+H(x_0,D\phi(x_0,t_0)), 
B(x_0,u(x_{0},t_{0}),D\phi(x_0,t_0))\}\le  0$$
{\rm (}resp.,
$$ 
\max\{\phi_t(x_0,t_0)+H(x_0,D\phi(x_0,t_0)), 
B(x_0,u(x_{0},t_{0}),D\phi(x_0,t_0))\}\ge 0).$$
We call u a solution of {\rm (IB)} if it is a subsolution and a supersolution of {\rm (IB)}.
\end{defn}

\subsection{Comparison Results for {\rm (IB)}}

\begin{thm}\label{thm:comparison}
Assume that {\rm (BA)} holds.\\
{\rm (i)} {\bf (Neumann/oblique derivative problem)}
Let $u\in C(\cQ)$, 
$v\in\LSC(\cQ)$ 
be a subsolution and a supersolution of {\rm (CN)}, respectively. 
If $u(\cdot,0)\le v(\cdot,0)$ on $\cO$, 
then $u\le v$ on $\cQ$. \\
{\rm (ii)} {\bf (State constraint or Dirichlet problem)}  
Let $u\in C(\cQ)$, $v\in\LSC(\cQ)$ 
be a subsolution and a supersolution of {\rm (SC)} 
or {\rm (CD)}, respectively. 
If $u(\cdot,0)\le v(\cdot,0)$ on $\cO$, 
then $u\le v$ on $\cQ$. 
\end{thm}

\begin{rem}
{\it 
It is worth mentioning that 
it is well known that there exists a discontinuous solution of 
{\rm (SC)} and {\rm (CD)}, 
which implies that 
in the comparison principle for {\rm (SC)} and {\rm (CD)}, 
we can not replace requirements of continuity of $u$  
by that of semicontinuity. 
}
\end{rem}

\begin{proof}[Proof of Theorem {\rm \ref{thm:comparison} \ (i)}]
We argue by contradiction assuming that there would exist $T>0$ such that $\max_{\ol{Q}_T}(u-v)(x,t)>0$, 
where $\QT:=\Om\times(0,T)$.

Let $u^{\del}$ denote the function
$$u^{\del}(x,t):=\max_{s\in[0,T+2]}\{u(x,s)-(1/\del)(t-s)^2\}\; ,$$ 
for any $\del>0$. This sup-convolution procedure is standard in the theory of viscosity solutions (although, here, it acts only on the time-variable) and it is known that, for $\del$ small enough, $u^{\del}$ is a subsolution of (CN) 
in $\Om\times(a_{\del},T+1)$, where $a_{\del}:=(2\del\max_{Q_{T+2}}|u(x,t)|)^{1/2}$ 
(see \cite{B2, CIL} for instance). 

Moreover, it is easy to check that 
$|u^{\del}_{t}|\le C_{\del}$ in $\Om\times(a_{\del},T+1)$ 
and 
therefore by the coercivity of $H$ and the $C^{1}$-regularity of 
$\bO$ we have, for all $x,y\in\cO$, $t,s\in[a_{\del},T+1]$
\begin{equation}\label{pf:comp:lip}
|u^{\del}(x,t)-u^{\del}(y,s)|\le 
C_{\del}(|x-y|+|t-s|) \ 
\end{equation}
for some $C_{\del}>0$.
Finally, as $\del\to 0$, $\max_{\ol{Q}_T}(u^{\del}-v)(x,t) \to \max_{\ol{Q}_T}(u-v)(x,t)>0$.

Therefore it is enough to consider $\max_{\ol{Q}_T}(u^{\del}-v)(x,t)$ for $\del>0$ small enough, and we follow the classical proof by introducing 
$$ \max_{\ol{Q}_T}\{(u^{\del}-v)(x,t)-\eta t\} \; ,$$
for $0< \eta \ll 1$. This maximum is achieved at $(\xi,\tau)\in\cO\times[0,T]$, namely
$(u^{\del}-v)(\xi,\tau)-\eta t
=\max_{\ol{Q}_T}\{(u^{\del}-v)(x,t)-\eta t\}$. 
Clearly $\tau$ depends on $\eta$ but we can assume that it remains bounded away from $0$, otherwise we easily get a contradiction.

We only consider the case where $\xi\in\bO$. 
We define the function $\Psi:\cO^{2}\times[0,T]\to\R$ by 
\begin{align*}
\Psi(x,y,t):=
&\, 
u^{\del}(x,t)-v(y,t)-\eta t
-\frac{1}{\ep^{2}}f(x-y)\\
{}&\, 
+g(\xi)(d(x)-d(y))
+
\al(d(x)+d(y))-|x-\xi|^{2}-(t-\tau)^{2}, 
\end{align*}
where $f$ and $d$ are given 
by \eqref{func-f} and \eqref{func-d}. 
Let $\Psi$ achieve its maximum at 
$(\ol{x},\ol{y},\ol{t})\in\cO^{2}\times[0,T]$. 
By a standard arguments, we have 
\begin{equation}\label{pf:comp:conv}
\ol{x}, \ol{y}\to \xi 
\ \textrm{and} \ 
\ol{t} \to \tau 
\ \textrm{as} \ \ep\to0 
\end{equation}
by taking a subsequence if necessary and, because of the Lipschitz continuity 
\eqref{pf:comp:lip} of $u^\del$, we have 
\begin{equation}\label{pf:comp:bdd}
\frac{|\ol{x}-\ol{y}|}{\ep^{2}}\le C_{\del} \ 
\end{equation}

Derivating (formally) $\Psi$ with respect to each variable 
$x,y$ at $(\ol{x},\ol{y},\ol{t})$,  we have 
\begin{align*}
D_{x}u^{\del}(\ol{x},\ol{t})
&=\, 
\frac{1}{\ep^{2}}A(\ol{x}-\ol{y})
+2(\ol{x}-\xi)-\al Dd(\ol{x})
-g(\xi)Dd(\ol{x}), \\
D_{y}v(\ol{y},\ol{t})&=\, 
\frac{1}{\ep^{2}}A(\ol{x}-\ol{y})
+\al Dd(\ol{y})
-g(\xi)Dd(\ol{y}). 
\end{align*}
We remark that 
we should interpret $D_{x}u^{\del}$ and $D_{y}v$ in the viscosity solution sense here. We also point out that the viscosity inequalities we are going to write down below, hold up to time $T$, in the spirit of \cite{B2}, Lemma 2.8, p. 41.

Since $\Om$ is a domain with a $C^{1}$-boundary, 
we first observe that  
for any $x\in\bO$ and $y\in\cO$ 
\[
(x-y)\cdot n(x)
\ge o(|x-y|), 
\]
where 
$o:[0,\infty)\to\R$ 
is a continuous function such that 
$|o(r)|/r\to0$ as $r\to0$.

Moreover, we have 
\begin{align*}
{}&
A(x-y)\cdot\gam(x)\\
\ge&\, 
(x-y)\cdot A\gam(\xi)
-|y-x|m_{\gam}(|x-\xi|)\\
=& \, 
(x-y)\cdot n(\xi)-|y-x|m_{\gam}(|x-\xi|)\\
\ge& \, 
(x-y)\cdot n(x)
-|y-x|\bigl(m_{\gam}(|x-\xi|)
+m_{n}(|x-\xi|)\bigr)\\
\ge& \, 
o(|x-y|)
-|y-x|\bigl(m_{\gam}(|x-y|)
+m_{n}(|x-\xi|)\bigr), 
\end{align*}
where 
$m_{\gam}$ and $m_{n}$ are modulus of continuity of 
$\gam, n$ on $\cO$, respectively.

Setting 
$m_{d}:=|o(r)|/r$, 
by \eqref{pf:comp:bdd}, we obtain 
\begin{equation}\label{pf-main:error-y}
\frac{1}{\ep^{2}}A(\ol{x}-\ol{y})\cdot \gam(\ol{x})
\ge 
-C_{\del}\bigl(
m_{d}(C_{\del}\ep^{2})
+m_{\gam}(C_{\del}\ep^{2})
+m_{n}(|\ol{x}-\xi|)\bigr). 
\end{equation}
Moreover, we have 
\[
\al Dd(\ol{x})\cdot\gam(\ol{x})
\le 
-\al+\al m_{\gam}(|\ol{x}-\xi|)
\]
and
\begin{align*}
g(\xi)Dd(\ol{x})\cdot\gam(\ol{x})
&\le \, 
-g(\ol{x})+m_{g}(|\ol{x}-\xi|), 
\end{align*}
where $m_{g}$ is a modulus of continuity of $g$ on $\bO$. 
Therefore we have 
\[
D_{x}u(\ol{x},\ol{t})\cdot \gam(\ol{x})
\ge 
g(\ol{x})-m(\ep)
+\al, 
\]
for a modulus $m$.  
Therefore, if $\ep>0$ is suitable small compared to $\al>0$, 
we have 
\[
D_{x}u(\ol{x},\ol{t})\cdot \gam(\ol{x})>g(\ol{x}). 
\]
Similarly, if $\ol{y}\in\bO$, then 
we have 
\[
D_{y}u(\ol{y},\ol{t})\cdot \gam(\ol{y})<g(\ol{y})
\]
for $\ep>0$ which is suitable small compared to $\al>0$.

Therefore, by the definition of viscosity solutions of (CN), using the arguments of the User's guide to viscosity solutions \cite{CIL}, there exists 
$a_{1}, a_{2}\in \R$ such that
\begin{align*}
a_{1}+H(\ol{x}, \frac{1}{\ep^{2}}A(\ol{x}-\ol{y})+2(\ol{x}-\xi)
-\al Dd(\ol{x})-g(\xi)Dd(\ol{x}))
&\, 
\le 0, \\ 
a_{2}+ H(\ol{y}, \frac{1}{\ep^{2}}A(\ol{x}-\ol{y})+\al Dd(\ol{x})
-g(\xi)Dd(\ol{y}))
&\, 
\ge0 
\end{align*}
with $a_{1}-a_{2}
=\eta+2(\ol{t}-\tau)$. 
By \eqref{pf:comp:bdd} 
we may assume that 
\[
\frac{1}{\ep^{2}}A(\ol{x}-\ol{y})\to p
\]
as $\ep\to0$ for some $p\in\R^{N}$ 
by taking a subsequence if necessary. 
Sending $\ep\to0$ and then $\al\to0$ 
in the above inequalities, 
we have a contradiction since $a-b \to \eta >0$ while the $H$-terms converge to the same limit. Therefore $\tau$ cannot be assumed to remain bounded away from $0$ and the conclusion follows.
\end{proof}

\begin{proof}[Proof of Theorem {\rm \ref{thm:comparison} \, (ii)}]
We only prove the comparison principle for (CD), 
since we can regard (SC) as a problem of (CD) 
with the extreme form ``$g(x)\equiv+\infty$".

To justify this choice, we first prove the

\begin{prop}\label{prop:classic} {\bf (Classical Dirichlet Boundary Conditions)}
Assume that {\rm (BA)} holds. 
If $u\in\USC(\cO\times[0,\infty))$ is a subsolution of 
{\rm (CD)}, then $u(x,t)\le g(x)$ 
for all $(x,t)\in\bQ$. 
\end{prop}

In the same way we can prove that, if $u\in\USC(\cO)$ is a subsolution of {\rm (D)$_{a}$} 
for any $a\ge c_{sc}$, then $u(x)\le g(x)$ for all $x\in\bO$. 

\begin{proof}
Fix $(x_0,t_0)\in\bO\times(0,\infty)$. 
Fix $r\in(0,t_0)$ and 
set $K:=\bigl(\ol{B}(x_0,r)\cap\cO\bigr)\times[t_0-r,t_0+r]$. 
Choose a sequence $\{x_k\}_{k\in\N}\subset\R^{N}\setminus\Om$ 
such that $|x_0-x_k|=1/k^2$. 
We define the function $\phi:K\to\R$ by 
\[\phi(x,t)=u(x,t)-k|x-x_k|-\alpha_k(t-t_0)^2,\]
where $\{\alpha_k\}_{k\in\N}\subset(0,\infty)$ is 
a divergent sequence which will be fixed later. 
Let $(\xi_k,\tau_k)\in K$ be a maximum point of $\phi$ on $K$. 
Noting that $\phi(\xi_k,\tau_k)\ge\phi(x_0,t_0)$, 
we have 
\begin{equation}\label{dirichlet-classic-estimate}
k|\xi_k-x_k|+\alpha_k(\tau_k-t_0)^2\le
u(\xi_k,\tau_k)-u(x_0,t_0)+k|x_0-x_k|\le C, 
\end{equation}
where $C>0$ is a constant independent of $k$. 
From the above, we see that 
$\xi_k\to x_0$, $\tau_k\to t_0$ as $k\to\infty$. 

By the viscosity property of $u$, we have
\[
q_k+H(\xi_k,p_k)\le0 
\qquad \textrm{or} \qquad 
u(\xi_k,\tau_{k})\le g(\xi_k), 
\]
if $k\in\N$ is sufficiently large, 
where $p_k=k(\xi_k-x_k)/|\xi_k-x_k|$ 
and $q_k=2\alpha_k(\tau_k-t_0)$. 
Set $f(k):=\min_{x\in B(x_0,r)\cap\cO}H(x,p_k)$. 
Noting that $|p_k|=k$, 
by the coercivity of $H$, 
$f(k)\to\infty$ as $k\to\infty$. 
Choose $\{\alpha_k\}_{k\in\N}\subset(0,\infty)$ 
such that 
\[\alpha_k\to\infty \ \textrm{as} \  k\to\infty 
\ \textrm{and} \ 
2\sqrt{C\alpha_k}+1\le f(k) \ 
\textrm{for sufficinetly large} \ k\in\N. \]
Since we have 
$\alpha_k|\tau_k-t_0|\le\sqrt{C\alpha_k}$ 
for all $k\in\N$ by \eqref{dirichlet-classic-estimate}, 
\[
q_k+H(x_k,p_k)
\ge-2\alpha_k|\tau_k-t_0|+f(k)
\ge-2\sqrt{C\alpha_k}+f(k)
\ge1>0
\]
for sufficiently large $k\in\N$, 
we must have 
$u(\xi_k,\tau_k)\le g(\xi_k)$. 
Sending $k\to\infty$, 
we obtain $u(x_0,t_0)\le g(x_0)$. 
\end{proof}

Now we return to the proof of the comparison result for (CD).
We argue by contradiction, exactly in the same way as for (CN), introducing $Q_{T}$ and
the function $u^{\del}$ which is a subsolution of (CD) in $\Om\times(a_{\del},T+1)$. We still denote by $(\xi,\tau)\in\cO\times[0,T]$, a maximum point of $(u^{\del}-v)(x,t)-\eta t$, namely
$$(u^{\del}-v)(\xi,\tau)-\eta \tau=\max_{\ol{Q}_T}\{(u^{\del}-v)(x,t)-\eta t\}\; .$$ 
Again we may assume that $\tau$ remains bounded away from $0$, otherwise the result would follow.

We only consider the case where $\xi\in\bO$. 
Since $\Om$ is a domain with a $C^{0}$-boundary, 
there exist $r>0$, $b\in C(\R^{N-1},\R)$ 
such that, 
after relabelling and re-orienting the coordinates axes 
if necessary, we may assume 
$B(\xi,r)\cap\Om=B(\xi,r)\cap\{(x',x_{N})\mid x_{N}>b(x')\}$. 
There exists $a_0>0$ such that 
for any $a\in(0,a_0)$, there exists a bounded open neighborhood $W_a$ 
of $\ol{B}(\xi,r)\cap\cO$ such that 
$W_a\subset -a e_{N}+\Om:=\{-a e_{N}
+x\mid x\in\Om\}$, 
where $e_{N}:=(0,\ldots,0,1)\in\R^{N}$. 

Since  
$(u^{\del}-g)(\xi,\tau)\le 0$ in view of Proposition \ref{prop:classic} 
and $u^{\del}(\xi,\tau)-v(\xi,\tau) >0$, 
we have $v(\xi,\tau)-g(\xi)<0$.

For any $a\in(0,a_0)$ and $(x,t)\in \ol{W}_a\times[0,T]$, 
set 
$u_{a}^{\del}(x,t):=u^{\del}(x+a e_{N},t)$. 
Then, it is easily seen that $u_{a}^{\del}$ satisfies 
\[(u_{a}^{\del})_t(x,t)+H(x+a e_{N},Du_{a}^{\del}(x,t))\le 0 \quad\textrm{in} \ W_a\times(a_\del,T+1)\]
in the viscosity sense.

Let $\ep>0$ and define the function 
$\Psi:\ol{W}_a\times(\ol{B}(\xi,r)\cap\cO)\times [0,T]
\to\R$ by 
\begin{align*}
\Psi(x,y,t,s)
:=&\, 
u_{a}^{\del}(x,t)-v(y,t)-\eta t \\
&-\frac{1}{2\ep^{2}}|x-y|^2-|y-\xi|^2-(t-\tau)^2.
\end{align*}

Set 
$V:=\ol{W}_a\times(\ol{B}(\xi,r)\cap\cO)\times[0,T]$ and 
let $(\ol{x},\ol{y},\ol{t})\in V$ be a maximum point of $\Psi$ 
on $V$. 
By the compactness of $V$ 
we may assume that 
$\ol{x},\ol{y}\to x_{a}$ and 
$\ol{t}\to t_{a}$
as $\ep\to0$ 
by taking a subsequence if necessary. 
Taking the limit as $\ep\to0$ in the inequality 
$\Psi(\ol{x},\ol{y},\ol{t})\ge \Psi(\xi,\xi,\tau)$
yields 
\begin{align*}
|x_{a}-\xi|^2+|t_{a}-\tau|^2&\le\, 
(u_{a}^{\del}-v)(x_{a},t_{a})-\eta t_{a}
-\{(u_{a}^{\del}-v)(\xi,\tau)-\eta \tau\}\\
{}
&\le\, 
(u^{\del}-v)(x_{a},t_{a})-\eta t_{a}
-\{(u^{\del}-v)(\xi,\tau)-\eta \tau\}\\
{}&\hspace*{13pt}+2m_{u}(a)\\
&\le\, 
2m_{u}(a), 
\end{align*}
where $m_{u}$ be a modulus of continuity of $u^{\del}$ 
on $\ol{Q}_{T}$. 
We notice that we use the continuity of $u$ here.

Fix a small $a\in(0,a_0)$ so that 
$v(x_{a},t_{a})-g(x_{a})<0$ and moreover 
we may assume that 
$v(\ol{y},\ol{t})-g(\ol{y})<0$ for $\ol{y}\in\bO$ if 
$\ep$ is sufficiently small.

By the viscosity property of $u_{a}^{\del}$ 
and $v$ and using classical arguments, there exist 
$b_{1}, b_{2} \in \R$ such that
\begin{gather}
b_1 
+H(\ol{x}+a e_{N},\frac{\ol{x}-\ol{y}}{\ep^{2}})\le 0, \label{comp4-sub}\\ 
b_2 
+H(\ol{y},\frac{\ol{x}-\ol{y}}{\ep^{2}}-2(\ol{y}-\xi))
\ge 0 \label{comp4-sup}
\end{gather}
with $d_1-d_2=\eta+2(\ol{t}-\tau)$. 
In view of the Lipschitz continuity of $u_{a}^{\del}$ 
on $\ol{W}_{a}\times[0,T]$, 
we have $(\ol{x}-\ol{y})/\ep^{2}$ is bounded 
uniformly $\ep>0$ 
and therefore, by sending $\ep\to0$  
we may assume 
$(\ol{x}-\ol{y})/\ep^{2}\to p_{\del,a}\in B(0,C_{\del,a})$ 
for some $C_{\del,a}>0$. 
Noting that $H$ is uniformly continuous on 
$\cO\times B(0,C_{\del,a}+1)$, 
we see that there exists a modulus $m_{\del}$ such that 
\begin{equation}\label{comp4-ineq}
H(x+a e_{N}, p)\ge H(x,p)-m_{\del}(a)\quad\textrm{for any} \ 
(x,p)\in \cO\times B(0,C_\del+1). 
\end{equation}

Combining \eqref{comp4-sub} and \eqref{comp4-sup} 
with \eqref{comp4-ineq}, 
if $a>0$ is small enough, we get 
\[
m_{\del}(a)\ge  
\eta+2(\ol{t}-\tau)
+H(\ol{x},\frac{\ol{x}-\ol{y}}{\ep^{2}})
-H(\ol{y},\frac{\ol{x}-\ol{y}}{\ep^{2}}-2(\ol{y}-\xi)).
\]
After taking the limit as $\ep\to0$, 
send $a\to0$ and then we get $\eta \leq 0$, 
which is a contradiction. 
We have thus completed the proof. 
\end{proof}

\subsection{Existence of Solutions of (IB)}

\begin{proof}[Proof of Theorem {\rm \ref{thm:existence}}]
We present the proof of the existence of Lipschitz continuous solutions of 
(IB).

Following similar arguments as in the proof of Theorem {\rm \ref{thm:additive}}, it is easy to prove that, for $C>0$ large enough $-Ct + u_0(x)$ and $Ct + u_0(x)$ are, respectively, viscosity subsolution and 
a supersolution of (CN) or (SC) or (CD) with, of course, different constant $C$ in each case.

By Perron's method (see \cite{I0}) and Theorem \ref{thm:comparison} 
we obtain continuous solutions of (IB) for (CN), (SC)  and (CD) that we denote by $u_{1}, u_{2}$ and $u_{3}$, respectively. 
As a consequence of Perron's method, we have 
$$ -Ct +u_0(x) \le u_i (x,t) \le Ct +u_0(x) \ \textrm{on} \ \cQ \; ,$$
for $i=1,2,3$. 

To conclude, we use a standard argument : comparing the solutions $u_i(x,t)$ and $u_i(x,t+h)$ for some $h>0$ and using the above property on the $u_i$, we have
$$ \|u_i(\cdot,\cdot+h)-u_i(\cdot,\cdot)\|_\infty 
\leq \|u_i(\cdot,h)-u_i(\cdot,0)\|_\infty \leq Ch\; .$$ 
As a consequence we have $\|(u_i)_t\|_\infty \leq C$ and, by using the equation together with (A1), we obtain that $Du$ is also bounded. This completes the proof. 
\end{proof}

We finally remark that 
we can deduce the uniform continuity of 
solutions of (IB) under the assumption $u_{0}\in C(\cO)$.  
We may choose a sequence 
$\{u_{0}^{k}\}_{k\in\N}\subset \W(\Om)\cap C(\cO)$ 
so that $\|u_{0}^{k}-u_{0}\|_{\Li(\Om)} \le 1/k$ 
for all $n\in\N$. 
Let $u_{k}$ be a solution of (IB) with $u_{0}=u_{0}^{k}$ and 
by the above argument we see $u_{k}\in\UC(\cQ)$ 
for all $k\in\N$.  
The maximum principle for (IB) implies that 
$u_{k}$ uniformly converges to $u$ on $\cQ$. 
Thus we obtain $u\in\UC(\cQ)$.

\medskip
\noindent
\textbf{Acknowledgements. }
Owing to a recent joint work \cite{BIM} 
with Hitoshi Ishii,  
the authors could refine the proof of Lemma \ref{lem:mu} 
in the case of Neumann problem 
on an early version of this paper.  
Moreover he gave the second author helpful suggestions 
on Sections 6.3 and Proposition 7.2. 
The authors are grateful to him. 
This work was partially done 
while the second author visited the Laboratoire de 
Math\'ematiques et Physique Th\'eorique, Universit\'e 
de Tours and Mathematics Department, 
University of California, Berkeley. 
He is grateful for their hospitality.


\bibliographystyle{amsplain}
\providecommand{\bysame}{\leavevmode\hbox to3em{\hrulefill}\thinspace}
\providecommand{\MR}{\relax\ifhmode\unskip\space\fi MR }
\providecommand{\MRhref}[2]{%
  \href{http://www.ams.org/mathscinet-getitem?mr=#1}{#2}
}
\providecommand{\href}[2]{#2}

\end{document}